\documentclass[12pt]{elsarticle}

\usepackage{amssymb}
\usepackage{lscape, amsmath,verbatim,graphicx,amsthm}
\usepackage{epsfig,float}
\usepackage{graphicx}
\usepackage{epsfig}
\usepackage{multirow}
\usepackage{lscape,amssymb, amsmath,verbatim,graphicx}
\usepackage{epsfig,float}
\usepackage{graphicx}
\usepackage{epsfig}
\usepackage[margin=1in]{geometry}

\newcommand{\tp}{t'}
\newcommand{\spr}{s'}

\newcommand{\beps}{{\mbox{\boldmath $\epsilon$}}}

\newcommand{\calN}{\mathcal N}

\newcommand{\la}{\langle}
\newcommand{\ra}{\rangle}

\newcommand{\Rn}{\mathbb R}
\newcommand{\intt}{\int\hspace{-.2cm}\int}

\newtheorem{remark}{Remark}[section]
\newtheorem{theorem}{Theorem}[section]
\newtheorem{lemma}{Lemma}[section]

\newcommand\X{{\bf X}}

\allowdisplaybreaks
\def\beq{\begin{equation}}
\def\eeq{\end{equation}}

\numberwithin{equation}{section}
\numberwithin{theorem}{section}

\begin{document}

\begin{frontmatter}

\title{On the asymptotic normality of kernel estimators of the long
run covariance of functional time series}
\tnotetext[mytitlenote]{Research supported by Austrian Science
Fund (FWF) Grant P2304-N18 and NSF grant DMS 1305858}



\author{Istv\'an Berkes}

\address{ Graz University of Technology,   Institute of Statistics, Kopernikusgasse 24, 8010 Graz, Austria.}

\author{Lajos Horv\'ath}

\author{Gregory Rice}


\ead{rice$@$math.utah.edu}
\address{Department of Mathematics, University of Utah, Salt Lake City, UT 84112--0090 USA}

\begin{keyword}
functional time series, long run covariance operator, normal approximation, moment inequalities, empirical eigenvalues and eigenfunctions
\end{keyword}

\begin{abstract}
We consider the asymptotic normality in $L^2$ of kernel
estimators of the long run covariance of stationary functional
time series. Our results are established assuming a weakly dependent
Bernoulli shift structure for the underlying observations, which contains
most stationary functional time series models, under mild conditions. As a
corollary, we obtain joint asymptotics for functional principal components
computed from empirical long run covariance operators, showing that they
have the favorable property of being asymptotically independent.
\end{abstract}


\end{frontmatter}

\section{Introduction}\label{intro}

In multivariate time series analysis, the matrix valued spectral density and the long run covariance matrix, which is $2\pi$ times the spectral density evaluated at frequency zero, are fundamental in a multitude of applications. For example, the long run covariance matrix must be estimated in most inference problems related to the mean of stationary finite dimensional time series, see e.g. Hannan (1970), Xiao and Wu (2012), Politis (2011), and Aue et al (2009). Additionally, dynamic principal component analysis utilizes estimates of the long run covariance matrix as well as the spectral density to perform meaningful dimension reduction for time series data, see Brillinger (2001).

Multivariate techniques are difficult to apply, however, when the data is obtained by observing a continuous time phenomena at a high resolution or at irregularly spaced time points. A flexible alternative for studying such records is to break them at natural points, for example into daily or monthly segments, in order to form a series of curves. The field of functional time series analysis has grown considerably in recent years to provide methodology for such data; the main difference from traditional functional data analysis being that it accommodates for possible serial dependence. The long run covariance kernel, which is an analog of the long run covariance matrix, also plays a crucial role in this setting. We refer to Ferraty and Vieu (2006) for a review of methods in functional data analysis and H\"ormann and Kokoszka (2012) for a survey on functional time series analysis.



In order to formally define the objects introduced above, let $\{X_i(t)\}_{i=-\infty}^{\infty}$, $t\in[0,1]$, be a stationary functional time series. The bivariate function

$$C(t,s)=\sum_{\ell=-\infty}^\infty\gamma_\ell(t,s) ,\quad \mbox{where  }\;\;\gamma_\ell(t,s)=\mbox{cov}(X_0(t), X_\ell(s)),$$

is called the long run covariance kernel, and is a well defined element of $L^2([0,1]^2,\Rn)$, assuming mild weak dependence conditions. $C(t,s)$ arises primarily as the asymptotic covariance of the sample mean function. Via right integration, $C(t,s)$ also defines a positive definite operator on $L^2([0,1],\Rn)$ whose eigenvalues and eigenfunctions, or principal components, are the focus of a number of dimension reduction and inference techniques with dependent functional data. Due to its representation as a bi--infinite sum, $C(t,s)$ is naturally estimated with a kernel lag--window estimator of the form

\begin{align}\label{estcoveq}
\hat{C}_{N}(t,s)=\sum_{i=-\infty}^{\infty}K\left( \frac{i}{h} \right) \hat{\gamma}_i(t,s),
\end{align}

where

\begin{displaymath}
\hat{\gamma}_i(t,s)=\hat{\gamma}_{i,N}(t,s)=
\left\{
\begin{array}{ll}
\displaystyle \frac{1}{N}\sum_{j=1}^{N-i}(X_j(t)-\bar{X}_N(t))(X_{j+i}(s)-\bar{X}_N(t)),\quad &i\geq 0\\
\vspace{.3cm}
\displaystyle \frac{1}{N}\sum_{j=1-i}^N(X_j(t)-\bar{X}_N(t))(X_{j+i}(s)-\bar{X}_N(t)), \quad &i<0.
\end{array}
\right.
\end{displaymath}

We use the standard convention that $\hat{\gamma}_i(t,s)=0$ when $i \ge N$.

The estimator in \eqref{estcoveq} was introduced in Horv\'ath et. al (2013), where it is shown to be consistent under mild conditions, and its applications are developed in Horv\'ath et al (2014) and Jirak (2013) in the context of inference for the mean and stationarity testing with functional time series. H\"ormann et al (2013) develops an analog of dynamic principal component analysis based on the spectral density operator of functional time series, which is directly related to the long run covariance operator.

It is a classical result that kernel lag--window estimators of the spectral density of univariate and multivariate time series are, when suitably standardized, asymptotically normal, see Rosenblatt (1991). The definition of the spectral density operator of a stationary functional time series and its asymptotic normality were first established in the work of Panaretos and Tavakoli (2013). In order to obtain their results, functional analogs of classical cummulant summability and mixing conditions are assumed. As noted in Shao and Wu (2007), cummulant conditions are exceedingly difficult to check, even with scalar time series, and mixing conditions, although classically popular, exhibit some unattractive pathologies. For example, the autoregressive one processes with independent and identically distributed errors that take the value 1 and -1 with equal probabilities are not mixing. On top of this, in several theaters of application non linear time series models are of interest, and in this case it is unknown whether such conditions are satisfied in the infinite dimensional setting.

%

In this paper we establish the asymptotic normality of $\hat{C}_N(t,s)$ in $L^2([0,1]^2,\Rn)$ for a broad class of stationary functional time series processes. In particular, we consider the case of $L^2([0,1],\Rn)$ valued random functions exhibiting an $L^p-m$ approximable Bernoulli shift structure, which extends the results of Shao and Wu (2007) and Liu and Wu (2010) to the infinite dimensional setting. Doing so greatly generalizes the class of functional time series processes for which a normal approximation for $\hat{C}_N$ can be achieved. An immediate corollary of this result is the limit distribution of the empirical eigenvalues and eigenfunctions computed from $\hat{C}_N$, which play a fundamental role in principal component analysis with dependent data.

The rest of the paper is organized as follows. In Section \ref{first} we state our assumptions and the main result of the paper. The section concludes with an application of our results to determining the optimal bandwidth parameter. Section \ref{eig-sec} contains the application to the limit distribution of the empirical eigenvalues and eigenfunctions computed from $\hat{C}_N$. The proofs of the main results of the paper are contained in Section \ref{proofsY}, which is broken into several subsections that each illuminate the main techniques behind the proof.

\section{Assumptions  and main results}\label{first}


 Let $||\cdot||$ denote the $L^2$ norm  of square integrable functions on $[0,1]^d$, the dimension  $d\geq 1$ being clear by the input function, and let $\int$ to mean $\int_0^1$. Throughout this paper we assume that
\begin{align}
&\X=\{X_i\}_{i=-\infty}^\infty \;\mbox{ forms a sequence of Bernoulli shifts, i.e. } X_j=g(\epsilon_j,\epsilon_{j-1},...)\label{as:1}\\
&\mbox{for some measurable function}\;\; g:S^\infty \mapsto L^2\;\;\mbox{and iid random variables} \; \epsilon_j,\;\notag \\
&-\infty<j<\infty,\;\mbox{with values in a  measurable space}\; S,\notag
\end{align}

\begin{align}\label{as:1/2}
\epsilon_j(t)=\epsilon_j(t,\omega)\;\;\mbox{is jointly measurable in}\;\;(t,\omega), \;-\infty<j<\infty
\end{align}
\begin{align} \label{as:2}
 E||X_0||^{4+\delta} < \infty \;\mbox{ for some\; $\delta >0$,}
\end{align}

and
\begin{align}\label{as:3}
&\{X_n\}_{n=-\infty}^\infty \;\;\mbox{can be approximated by the } \; m\mbox{--dependent}\;\mbox{sequences}\\
&X_{n,m}=g(\epsilon_n,\epsilon_{n-1},...,\epsilon_{n-m+1},\beps_{n,m}^*) \;\mbox{with   } \beps^*_{n,m}=(\epsilon_{n,m, n-m}^*,\epsilon_{n,m, n-m-1}^*,\ldots), \notag\\
&\mbox{where the }\; \epsilon^*_{n,m,k}\mbox{'s} \;\mbox{are independent copies of}\; \epsilon_0, \mbox{independent of}  \notag\\
& \{\epsilon_i, -\infty<i<\infty\},\; \mbox{such that}\; \notag\\
&\hspace{2 cm}\sum_{k=1}^\infty w\left(\sum_{m=k}^\infty c_m\right)<\infty\;\;\;\mbox{with}\;\;c_m=(E\|X_0-X_{0,m}\|^4)^{1/4}, \notag\\
&\mbox{where }  w(t)>0 \mbox{ is regularly varying at zero, and } w(t)/t^{1/3}\to 0 \notag\\  
&\mbox{as } t\to 0. \notag
\end{align}

Nearly all stationary time series models based on independent innovations satisfy condition \eqref{as:1}, including linear processes in function spaces, and the functional ARCH and GARCH processes, see Bosq (2000) and H\"ormann et al (2013). Condition (\ref{as:3}) specifies the level of dependence that is allowed within the sequence in terms of how well it can be approximated in the $L^2$ sense by finite dependent processes, and thus defines a version of $L^p$--$m$--approximability for functional time series, see H\"ormann and Kokoszka (2010). Condition (\ref{as:3}) is satisfied when, for example, $c_m=O(m^{-\alpha})$ for
some $\alpha>4$. In comparison, Shao and Wu (2007) assume a geometric rate of decay for similar approximation coefficients with scalar time series. It follows from \eqref{as:1}--\eqref{as:3} that $C(t,s)$ is an element of $L^2([0,1]^2,\Rn)$ (cf.\ Appendix A.2 in Horv\'ath et al (2013)).

We assume that the kernel $K$ in the definition of $\hat{C}_N$ satisfies the following standard conditions:
\beq\label{k-1}
K(0)=1,
\eeq
\beq\label{k-2}
K\;\;\mbox{is symmetric around }0,\;\;K(u)=0\;\;\mbox{if}\;\; u>c\;\;\mbox{with some } c>0,
\eeq
and
\beq\label{k-3}
K\;\;\mbox{is Lipschitz  continuous on}\;\;[-c,c], \;\mbox{where}\; c \;\mbox{is given in}\; \eqref{k-2}.
\eeq
Lastly we take the window (or smoothing parameter) $h$ to satisfy that
\beq\label{h-1}
h=h(N)\to\infty\;\;\;\mbox{and}\;\;\;\frac{h(N)}{N}\to 0,\;\;\;\mbox{as}\;\;\;N\to \infty.
\eeq

The main result of our paper establishes the asymptotic limit distribution  of
$$
Z_N(t,s)=\hat{C}_N(t,s)-E\hat{C}_N(t,s).
$$

\begin{theorem}\label{main-1} If \eqref{as:1}--\eqref{h-1} and

\beq\label{m-1}
h/N^{\delta/(4+2\delta)}\to 0
\eeq

hold, where $\delta$ is defined by \eqref{as:2}, then one can define a sequence of Gaussian processes  $\Gamma_N(t,s)$ defined on the same probability space, and satisfying $E\Gamma_N(t,s)=0$, $E\Gamma_N(t,s)\Gamma_N(\tp,\spr)=L(t,s,\tp,\spr)$ with \beq\label{L-def}
L(t,s,\tp, \spr)=\left[C(t,s)C( \tp, \spr) +C(t,\tp)C(s,\spr) \right] \int_{-c}^c K^2(z)dz
\eeq
such that
\beq\label{m-0}
\|(N/h)^{1/2}Z_N-\Gamma_N\|\stackrel{P}{\to}0,\quad \mbox{as}\;\;N\to \infty.
\eeq
\end{theorem}

Theorem \ref{main-1} provides a Skorokhod--Dudley--Wichura representation of the weak convergence of $Z_N$ to a Gaussian process. We note that the distribution of the limiting Gaussian process $\Gamma_N$ does not depend on $N$, and hence the approximation \eqref{m-0} can be readily used to compute the limiting behavior of functionals of $Z_N$.

Assuming a higher moment condition than \eqref{as:2}, a more optimal result can be proven in the sense that we can relax \eqref{m-1} in Theorem \ref{main-1}.

\medskip

\begin{theorem}\label{main-2} If \eqref{as:1}--\eqref{h-1} and
\beq\label{m-2}
E\|X_0\|^8<\infty
\eeq
are satisfied, then \eqref{m-0} holds.
\end{theorem}

\medskip

\begin{remark}{\rm Panaretos and Tavakoli (2013) and H\"ormann et al (2013) consider the estimation and, in the case of Panaretos and Tavakoli (2013), the asymptotic theory of more general objects that they refer to as the spectral density kernels; they are defined by
$$
f_{\omega}(t,s) = \frac{1}{2\pi} \sum_{j \in {\mathbb Z}} exp(- {\bf i} \omega j)\gamma_j(t,s), \;\; \omega \in [0, 2\pi),
$$
where ${\bf i}$ is the imaginary unit. For a fixed $\omega$, $f_\omega$ is estimated analogously to the long run covariance kernel by
$$
\hat{f}_{\omega}(t,s)= \frac{1}{2\pi} \sum_{j \in {\mathbb Z}}K\left(\frac{j}{h}\right) exp(- {\bf i} \omega j) \hat{\gamma}_j(t,s).
$$
In this paper we only consider the asymptotics of $2\pi \hat{f}_0(t,s)$, but we could extend our results to the case of the joint asymptotics of $\hat{f}_\omega$ over $\omega$ as in Panaretos and Tavakoli (2013). This would require working with the fourier transform component of the definition and follows along the lines of the univariate case as demonstrated in Brillinger (2001). The difficult and novel part of this theory is in establishing a Gaussian approximation for the function space component of the spectral density, and this is captured by the long run covariance kernel.
}
\end{remark}

\subsection{Bias, bandwidth selection, and positive definiteness}

In order to infer from this the limit behavior of $\hat{C}_N - C$, we
must also consider the bias. Following Parzen (1957),  we assume that
\begin{align}\label{par-1}
&\mbox{there exists a  }{\frak q}>0\mbox{  such that  }0<\lim_{x\to 0}\frac{K(x)-1}{|x|^{\frak q}}={\frak K}<\infty,
\end{align}
and
\beq\label{par-2}
\mbox{there exists a   }{\frak q}'>{\frak q}\mbox{  such that  }\sum_{\ell=-\infty}^\infty |\ell|^{{\frak q}'}\|\gamma_\ell\|<\infty.
\eeq

The asymptotic bias is given by $h^{-{\frak q}}{\frak F}(t,s)$, where
$$
{\frak F}(t,s)={\frak K}\sum_{\ell=-\infty}^{\infty}|\ell|^{\frak q}\gamma_\ell(t,s).
$$

\begin{theorem}\label{bias-1} If \eqref{par-1}, \eqref{par-2} hold and $h^{\frak q}/N\to 0$,  then we have
$$
\left\| E\hat{C}_N- C-h^{-{\frak q}}{\frak F}\right\|=o(h^{-{\frak q}}).
$$
\end{theorem}

\medskip
We note that if the unbiased estimators $N\hat{\gamma}_{i,N}(t,s)/(N-i)$ are used in the definition of $\hat{C}_N$, then Theorem \ref{bias-1} remains true without assuming $h^{\frak q}/N\to 0$. \\

The minimization of the asymptotic mean squared error provides a popular choice for $h$ in case of univariate data (cf.\ Parzen (1957) and Andrews (1991)). In our case the ``optimal" $h$ minimizes $E\|\hat{C}_N-C\|^2$. Our results show that
\beq\label{m-fro}
E\|\hat{C}_N-C\|^2\approx \frac{h}{N}E\|\Gamma_1\|^2+h^{-2{\frak q}}\|{\frak F}\|^2.
\eeq
Since
$$
E\|\Gamma_1\|^2=2\left(\intt C(t,s)dtds\right)^2\int_{-c}^cK^2(u)du,
$$
we get that the minimum of the asymptotic value of the mean squared error in \eqref{m-fro} is reached  at
$$
h_{{\it opt}}\approx {\frak c}_0N^{1/(1+2{\frak q})},
$$
where
$$
{\frak c}_0=\left({\frak q}\|{\frak F}\|^2\right)^{1/(1+2{\frak q})}\left(\left(\intt C(t,s)dtds\right)^2\int_{-c}^cK^2(u)du   \right)^{-1/(1+2{\frak q})}.
$$
The constant ${\frak c}_0$ is a complicated function of the unknown correlations $\gamma_\ell(t,s)$ and the long run covariance function $C(t,s)$. Replacing the unknown functions $\gamma_\ell(t,s)$ and $C(t,s)$ with their empirical counterparts, we get a plug in estimate for ${\frak c}_0$. A data driven estimator is discussed in Horv\'ath et al (2014) for the ``flat top" kernel, i.e.\ when ${\frak q}=\infty$. 

According to \eqref{m-fro} and since $h$ tends to infinity with $N$, the asymptotic integrated mean squared error is minimized by using a kernel $K$ for which ${\frak q}$ may be taken to be as large as possible. This encourages the use of a kernel function that is smooth or ``flat" near the origin, but for arbitrary kernels $\hat{C}$ need not be positive definite. Several methods have been proposed to address related issues in the finite dimensional setting, see Politis (2011); they typically involve either sacrificing possible improvements in the bias by using a kernel that makes the estimator positive definite from the outset, like the Bartlett kernel, or using a higher order kernel and then altering the estimator to be positive definite by removing the negative eigenvalues from the diagonalization of the operator. These methods could be adapted to the functional setting, and the authors plan on studying such techniques in future work.

\section{Application to the limit distribution of functional principal components}\label{eig-sec}

A technique to reduce the dimension of functional data that has received considerable attention, both in applications and theoretical investigations, is  principal component analysis (PCA); we refer to Ramsey and Silverman (2005) and Horv\'ath and Kokoszka (2012) for reviews of the subject. Typically the principal components used are computed as the eigenfunctions of the  sample covariance function
$$
\hat{C}^s_N(t,s) = \frac{1}{N}\sum_{i = 1}^N (X_i(t)-\bar{X}_N(t)) (X_i(s)-\bar{X}_N(s)).
$$
Due to their important role in PCA, the difference between the empirical and theoretical eigenvalues and eigenfunctions have been investigated by several authors. Kokoszka and Horv\'ath (2012, pp.\ 31--35) contains inequalities for the accuracy of the replacement of the theoretical PCA's with their empirical counterparts. The asymptotic normality of the deviation between the empirical and theoretical eigenvalues and eigenfunctions was proven by Dauxois et al (1982), Bosq (2000) and Hall and Hosseini--Nasab (2007) assuming that the $X_i$'s are independent and identically distributed. In great generality, Mas and Menneteau (2003) show that the asymptotic properties of the empirical eigenvalues and eigenfunctions are automatically inherited from the asymptotic properties of their corresponding operators. Kokoszka and Reimherr (2012) investigated the asymptotic properties of the eigenvalues and eigenfunctions of $\hat{C}^s_N$ when the observations are from a stationary functional time series.   \\

In case of inference with dependent functional data it may be preferable to use the theoretical principal components $\{v_i\}_{i\ge1}$ defined by the the long run covariance operator,

\begin{align}\label{theoevals}
\lambda_i v_i(t)=\int C(t,s)v_i(s)ds,
\;\;1\leq i < \infty,
\end{align}

where we have used $\lambda_1 \ge \lambda_2 \ge \cdots \ge 0$ to denote the ordered eigenvalues. These define an example of dynamic functional principal components as defined in Hormann et al (2014). The theoretical eigenvalues and eigenfunctions defined in \eqref{theoevals} can be estimated from a sample by the eigenvalues and eigenfunctions of the empirical long run covariance function
\begin{align}\label{evalest}
\hat{\lambda}_i\hat{v}_i(t)=\int \hat{C}_N(t,s)\hat{v}_i(s)ds,
\;\;1\leq i\leq N.
\end{align}
It was shown in Horv\'ath et al (2012) that if for some $p\ge 1$,
\begin{align}\label{evals}
\lambda_1 > \lambda_2 > \cdots > \lambda_p >\lambda_{p+1} \ge 0,
\end{align}
then the estimators defined in \eqref{evalest} are asymptotically consistent in the sense that
$$
\max_{1 \le i \le p } |\hat{\lambda}_i - \lambda_i | =o_P(1), \;\;\mbox{and}\;\; \max_{1 \le i \le p } \| \hat{s}_i\hat{v}_i - v_i \| =o_P(1),\quad \mbox{as}\;\;N\to \infty,
$$
where $\hat{s}_i=\mbox{sign}( \langle \hat{v}_i, v_i \rangle )$. We show that Theorems \ref{main-1}--\ref{bias-1} imply the limit distributions of $(N/h)^{1/2}(\hat{\lambda}_i-\lambda_i)$, $(N/h)^{1/2}(\hat{v}_i(t)-v_i(t))$ and $(N/h)^{1/2}\|\hat{v}_i-v_i\|$, $1\leq i\leq p.$

\begin{theorem}\label{th:eig-1} Under   the conditions of Theorems \ref{main-1} and \ref{bias-1}, \eqref{evals}  and  assuming $\lim_{N\to \infty} N/h^{1+2q}= {\mathfrak a}$, there exist random variables ${\frak g}_{\ell,N}$ and random functions ${\mathcal  G}_{\ell,N}(t), 1\leq \ell\leq p$ such that
$$
\max_{1\leq \ell \leq p}|\left({{N}}/{{h}}\right)^{1/2} (\hat{\lambda}_\ell - \lambda_\ell)-{\frak g}_{\ell,N}|=o_P(1),
$$

$$
\max_{1\leq \ell \leq p}\left\|(N/h)^{1/2}( \hat{s}_\ell \hat{v}_\ell-v_\ell) -{\mathcal G}_{\ell,N} \right\|=o_P(1)
$$

and

\begin{align*}
&\left\{ {\frak g}_{\ell,N},  {\mathcal G}_{\ell,N}(t), 1\leq \ell\leq p\right\}\\
&\hspace{1 cm}\stackrel{{\mathcal D}}{=}\Biggl\{
\lambda_\ell\left(2\int_{-c}^cK^2(z)dz\right)^{1/2} {\mathcal N}_{\ell, \ell}+{\frak a}\intt{\frak F}(u,s)v_\ell(s)v_\ell(u)duds,\\
&\hspace{2.5cm}\left(\int_{-c}^cK^2(z)dz\right)^{1/2}\sum_{1\leq k\neq\ell<\infty}v_k(t)\frac{(\lambda_\ell\lambda_k)^{1/2}}{\lambda_\ell-\lambda_k}{\mathcal N}_{\ell,k}\\
&\hspace{3.5cm}+{\frak a}
\sum_{1\leq k\neq\ell<\infty}\frac{v_k(t)}{\lambda_\ell-\lambda_k}\intt {\frak F}(u,s)v_\ell(u)v_\ell(s)duds, 1\leq \ell\leq p
\Biggl\},
\end{align*}
where ${\mathcal N}_{\ell, k}$ are independent standard normal random variables.
\end{theorem}

If ${\frak a}=0$, then Theorem \ref{th:eig-1} implies that the empirical eigenvalues and eigenfunctions are asymptotically consistent with rate $(h/N)^{1/2}$. Theorem \ref{th:eig-1} also shows  that $({{N}}/{{h}})^{1/2} (\hat{\lambda}_\ell - \lambda_\ell)$ are asymptotically independent and normally distributed, and that, on top of being orthogonal functions, $ (N/h)^{1/2}( \hat{s}_\ell \hat{v}_k-v_k)\;\; 1\le k \le p$ are stochastically independent and Gaussian. This result is along the lines of the asymptotic independence and normality of the suitably normed and centered empirical eigenvalues and  eigenfunctions of the empirical covariance function of independent and identically distributed functional observations.  The main difference is the norming; we use $(N/h)^{1/2}$ in the case of the kernel estimator for the long run covariance function instead of the $N^{1/2}$ rate in the  case of the sample covariance.
Since   $(N/h)^{1/2}\|\hat{s}_\ell \hat{v}_\ell-v_\ell\|^2, 1\leq \ell \leq p$ are  asymptotically independent, assuming that ${\frak a}=0$,  Theorem \ref{th:eig-1} yields

\begin{align*}
\frac{N}{h}\|\hat{s}_\ell \hat{v}_\ell-v_\ell\|^2\;\;\stackrel{{\mathcal D}}{\to}\;\; \lambda_\ell\int_{-c}^cK^{2}(z)dz\sum_{k\neq \ell}\frac{\lambda_k}{(\lambda_\ell-\lambda_k)^2}{\mathcal N}_{\ell, k}^2.
\end{align*}

This is the analogue of the result of Dauxois et al (1982) to  the functional time series case.

\section{Proofs of Theorems \ref{main-1}--\ref{th:eig-1}}\label{proofsY}

The proofs of the main results of the paper, Theorems \ref{main-1} and Theorems \ref{main-2}, are carried out in three primary steps. Firstly, we show that the process $Z_N$ can be well approximated by an analogous process $Z_{N,m}$ that is constructed using $m$--dependent random functions with the aid of \eqref{as:3} in Subsection \ref{lem-mdep}. Once we have achieved this approximation, we obtain a lower dimensional approximation $Z^d_{N,m}$ based on $d$ dimensional random functions via a projection technique in Subsection \ref{fine-dim}. It is then straightforward to create a gaussian approximation for this process (Subsection \ref{clt-fine-dim}), and we may then retrace our steps with the gaussian process by letting $d$ and $m$ tend to infinity (Subsection \ref{proof-1}). We then obtain as a simple corollary the asymptotic distributions of the eigenvalues and eigenfunctions in Subsection \ref{sec-eig-pr}.

\subsection{Approximation with $m$--dependent sequences}\label{lem-mdep}

To simplify notation, we assume throughout the proofs that $c=1$ in \eqref{k-2} and \eqref{k-3}.\\
First we show that replacing the sample mean $\bar{X}_N(t)$ with $EX_0(t)$ in the definition of $\hat{\gamma}_i$ does not effect the limit distribution of $Z_N^2.$ It is clear that  we can assume without loss of generality that
\beq\label{zzero}
EX_i(t)=0.
\eeq
Let
\begin{align*}
\tilde{C}_N(t,s)=\sum_{i=-\infty}^\infty K\left(\frac{i}{h}\right)\tilde{\gamma}_i(t,s),
\end{align*}
where
\begin{align}\label{til-g}
\tilde{\gamma}_i(t,s)=\tilde{\gamma}_{i,N}(t,s)=
\left\{
\begin{array}{ll}
\displaystyle \frac{1}{N}\sum_{j=1}^{N-i}X_j(t)X_{j+i}(s),\quad &i\geq 0\\
\vspace{.3cm}
\displaystyle \frac{1}{N}\sum_{j=1-i}^N X_j(t)X_{j+i}(s), \quad &i<0.
\end{array}
\right.
\end{align}
We prove in the following lemma that $Z_N$ and $\tilde{Z}_N$ have the same limit distribution, where
$
\tilde{Z}_N(t,s)=\tilde{C}_N(t,s)-E\tilde{C}_N(t,s).
$
\begin{lemma}\label{m-rem} If \eqref{as:1}--\eqref{h-1} are satisfied, then we have that
\beq\label{mr-1}
\frac{N}{h}\|Z_N-\tilde{Z}_N\|^2=o_P(1).
\eeq
\end{lemma}

\begin{proof} It is easy to see that

\begin{align}\label{mea-1}
&\|Z_N-\tilde{Z}_N\|\\
&\hspace{.3 cm}\leq \|\bar{X}_N\|\left\{\left\|\sum_{i=0}^\infty K\left(\frac{i}{h}\right)\frac{1}{N}\sum_{j=1}^{N-i}X_j\right\|+\left\|\sum_{i=0}^\infty K\left(\frac{i}{h}\right)\frac{1}{N}\sum_{j=i+1}^{N}X_j\right\|\right\}\notag\\
&\hspace{1 cm}+\|\bar{X}_N\|\left\{\left\|\sum_{i=-\infty}^0 K\left(\frac{i}{h}\right)\frac{1}{N}\sum_{j=1-i}^{N}X_j\right\|+\left\|\sum_{i=-\infty}^0 K\left(\frac{i}{h}\right)\frac{1}{N}\sum_{j=1}^{N+i}X_j\right\|\right\}\notag\\
&\hspace{1 cm}+\|\bar{X}_N\|^2\left|  \sum_{i=-\infty}^\infty K\left(\frac{i}{h}\right)  \right|.\notag
\end{align}
Berkes et al (2013) showed that under \eqref{as:1}--\eqref{as:3}
\beq\label{bhr-1}
\|\bar{X}_N\|=O_P(N^{-1/2}),
\eeq
 and therefore by \eqref{k-2} and   \eqref{k-3}
 $$
\| \bar{X}_N\|^2\left|  \sum_{i=-\infty}^\infty K\left({i}/{h}\right)  \right|=O_P(h/N).
 $$
On account of $EX_0(t)X_{i,i}(s)=0$, by \eqref{as:3} we have that
\begin{align}\label{simple-2}
\sum_{ i=1}^\infty\left|\intt E X_0(t)X_i(s)dtds\right|\leq (E\|X_0\|^2)^{1/2}\sum_{i=1}^\infty (E\|X_0-X_{0,i}\|^2)^{1/2}<\infty,
\end{align}
and therefore we obtain immediately that
 \begin{align*}
 &\left\|\sum_{i=0}^\infty K\left({i}/{h}\right)\frac{1}{N}\sum_{j=1}^{N-i}X_j\right\|^2=N^{-2}\sum_{i,\ell=0}^\infty K\left({i}/{h}\right) K\left({\ell}/{h}\right)\sum_{j=1}^{N-i}\sum_{k=1}^{N-\ell}\intt
 EX_j(t)X_k(s)dtds\\
 &\hspace{1cm}=O(1/N)\sum_{i,\ell=0}^\infty| K\left({i}/{h}\right) K\left({\ell}/{h}\right)|\\
 &\hspace{1cm}=O_P(h^2/N),
 \end{align*}
 where we used again \eqref{k-2} and   \eqref{k-3}. Thus we get by \eqref{bhr-1} that
 $$
 \|\bar{X}_N\|\left\|\sum_{i=0}^\infty K\left({i}/{h}\right)\frac{1}{N}\sum_{j=1}^{N-i}X_j\right\|=O_P(h/N).
 $$
 Similar arguments provide the same upper bounds for the other terms in \eqref{mea-1} which implies that $(N/h)^{1/2}\|\tilde{Z}_N-Z_N\|=O_P((h/N)^{1/2})=o_P(1) $ which gives \eqref{mr-1} .
\end{proof}

\medskip
We recall $X_{n,m}, 1\leq n \leq N$  defined in \eqref{as:3} for $m\geq 0$. 
Replacing $X_i$ with $X_{i,m}$ in the definition of $\tilde{C}_N$ we define
\begin{align*}
\tilde{C}_{N,m}(t,s)=\sum_{i=-\infty}^\infty K\left(\frac{i}{h}\right)\tilde{\gamma}_i^{(m)}(t,s),
\end{align*}
where
\begin{displaymath}
\tilde{\gamma}^{(m)}_i(t,s)=
\left\{
\begin{array}{ll}
\displaystyle \frac{1}{N}\sum_{j=1}^{N-i}X_{j,m}(t)X_{j+i,m}(s),\quad &i\geq 0\\
\vspace{.3cm}
\displaystyle \frac{1}{N}\sum_{j=1-i}^N X_{j,m}(t)X_{j+i,m}(s), \quad &i<0
\end{array}
\right.
\end{displaymath}
and
$$
\tilde{Z}_{N,m}(t,s)=\tilde{C}_{N,m}(t,s)-E\tilde{C}_{N,m}(t,s).
$$
\begin{lemma}\label{f-lem} If \eqref{as:1}--\eqref{h-1} are satisfied,  then we have
\beq\label{fi-1}
\lim_{m \to \infty}\| C_m-C\|=0
\eeq
where
$$
C_m(t,s)=\sum_{\ell=-m}^{m}EX_{0,m}(t)X_{\ell,m}(s).
$$
Also, as $m\to \infty$,
\beq\label{fi-2}
 \int C_m(t,t)dt\;\to\;\int C(t,t)dt,
\eeq
\beq\label{fi-3}
\intt \left(  \sum_{\ell=-\infty}^\infty EX_0(t)X_{\ell,m}(s)\right) \left(\sum_{j=-\infty}^\infty E X_{0,m}(t)X_\ell(s)\right)dtds \;\to\;\| C\|^2,
\eeq
and
\beq\label{fi-4}
\int \sum_{\ell=-\infty}^\infty EX_0(t)X_{\ell, m}(t)dt \;\to\;\int C(t,t)dt,
\eeq
\end{lemma}
\begin{proof} By definition we have
$$
C(t,s)=\sum_{\ell=-\infty}^{-m-1}EX_0(t)X_\ell(s)+\sum^{\infty}_{\ell=m+1}EX_0(t)X_\ell(s)+\sum^{m}_{\ell=-m}EX_0(t)X_\ell(s).
$$
Due to the fact that $C(t,s)$ is in $L^2([0,1]^2,\Rn)$, it follows that
\begin{align*}
\left\|\sum^{\infty}_{\ell=m+1}EX_0(t)X_\ell(s)\right\|\to 0,  \quad \mbox{as}\quad m\to \infty
\end{align*}
and
$$
\left\| \sum_{\ell=-\infty}^{-m-1}EX_0(t)X_\ell(s)   \right\|\to \infty, \quad \mbox{as}\quad m\to \infty.
$$
Clearly,
\begin{align*}
EX_0(t)X_\ell(s)&-EX_{0,m}(t)X_{\ell,m}(s)\\
&=EX_{0}(t)X_\ell(s) -  EX_{0,m}(t)X_\ell(s)+EX_{0,m}(t)X_\ell(s)-EX_{0,m}(t)X_{\ell,m}(s)
\end{align*}
and therefore by \eqref{as:3} and stationarity we conclude
$$
\|EX_0(t)X_\ell(s)-EX_{0,m}(t)X_{\ell,m}(s)\|\leq 2 (E\|X_0\|^2E\|X_0-X_{0,m}\|^2)^{1/2}.
$$
Hence
\begin{align*}
\left\|\sum^{m}_{\ell=-m}(EX_0(t)X_\ell(s)-EX_{0,m}(t)X_{\ell,m}(s))\right\|\leq 2(2m+1)(E\|X_0\|^2E\|X_0-X_{0,m}\|^2)^{1/2}\to 0,
\end{align*}
as $m\to \infty$, completing the proof of \eqref{fi-1}.
Similar arguments give \eqref{fi-2}.\\
To prove \eqref{fi-3} we first define
$$ r_{1,m}=\{|\ell|>m, |j|\leq m\}, \;\;r_{2, m}=\{|\ell|\leq m, |j|>m\},\;\;   r_{3, m}=\{|\ell|> m, |j|>m\}
$$
and
$$r_{4,m}=\{|\ell|, |j|\leq m\}.
$$
 Let
$$
\alpha_{\ell, j, m}(t,s)=EX_0(t)X_{\ell, m}(s)EX_{0,m}(t)X_j(s)-a_\ell(t,s)a_j(t,s),\;\mbox{where   }a_\ell (t,s)=EX_0(t)X_\ell(s).
$$
 For all $\ell>m$ we have that $EX_0(t)X_{\ell,m}(s)=0$ and therefore by the Cauchy--Schwarz inequality
\begin{align*}
\left|\intt \sum_{r_{1,m,1}}\alpha_{\ell, j, m}(t,s)dtds\right|& = \left|  \intt \sum_{r_{1,m,1}}   a_\ell(t,s)a_j(t,s) dtds \right|\\
&\leq 2 E\|X_0\|^2\sum_{\ell>m}^\infty(E\|X_0-X_{0,\ell}\|^2)^{1/2}\sum_{j=0}^\infty(E\|X_0-X_{0,j}\|^2)^{1/2}\\
&\to 0, \quad \mbox{as}\;\;m\to \infty,
\end{align*}
where $r_{1,m,1}=\{\ell>m, |j|\leq m\}$. On the set $r_{1,m,2}=\{\ell<-m, |j|\leq m\}$ we write by the independence of $X_{0,\ell}$ and $X_{\ell, m}$ that
\begin{align*}
\sum_{r_{1,m,2}}\intt &|EX_0(t)X_{\ell,m}(s)EX_{0,m}(t)X_j(s)|dtds\\
&= \sum_{r_{1,m,2}}\intt |E(X_{0,m}(t)-X_{-\ell}(t))X_{\ell,m}(s)EX_{0,m}(t)X_j(s)|dtds\\
&\leq \sum_{r_{1,m,2}}(E\|X_0-X_{0,-\ell}\|^2)^{1/2}(E\|X_0\|^2)^{3/2}\\
&\leq (2m+1)(E\|X_0\|^2)^{3/2} \sum_{\ell=m}^\infty(E\|X_0-X_{0,\ell}\|^2)^{1/2}.
\end{align*}
It follows similarly that
$$
\intt \sum_{r_{1,m,2}}  | a_\ell(t,s)a_j(t,s)| dtds \leq (2m+1) (E\|X_0\|^2)^{3/2} \sum_{\ell=m}^\infty(E\|X_0-X_{0,\ell}\|^2)^{1/2}
$$
resulting in
$$
\left|\intt \sum_{r_{1,m}}\alpha_{\ell, j, m}(t,s)dtds\right|\to 0, \;\;\;\mbox{as}\;\;m\to \infty
$$
via \eqref{as:3}. Similar arguments give for $i=2,3,4$  that
$$
\left|\intt_{r_{i,m}}\alpha_{\ell, j, m}(t,s)dtds\right|\to 0, \quad \mbox{as}\;\;m\to \infty.
$$
Observing that
$$
\intt \sum_{\ell, i=-\infty}^\infty a_\ell(t,s)a_i(t,s)dtds=\|C\|^2,
$$
the proof of \eqref{fi-3} is complete. The proof of \eqref{fi-4} goes along the lines of \eqref{fi-3}.
\end{proof}

\medskip
 \begin{lemma}\label{sec-lem}
  If \eqref{as:1}--\eqref{h-1} are satisfied, then we have
 \beq\label{sec-1}
 \lim_{m\to \infty}\limsup_{N\to \infty}\frac{N}{h}E\intt (\tilde{Z}_N(t,s)-\tilde{Z}_{N,m}(t,s))^2dtds=0.
 \eeq
 Also, for each $m\geq 1$
 \beq\label{sec-2}
\lim_{N\to \infty}\frac{N}{h}\intt \mbox{\rm var}(\tilde{Z}_N(t,s))dtds=\left(\|C\|^2 +\left(\int C(t,t)dt\right)^2 \right)\int_{-1}^1 K^2(u)du,
 \eeq
 \beq\label{sec-3}
\lim_{N\to \infty}\frac{N}{h}\intt \mbox{\rm var}(\tilde{Z}_{N,m}(t,s))dtds  =\left(\|C_m\|^2 +\left(\int C_m(t,t)dt\right)^2 \right)\int_{-1}^1 K^2(u)du  ,
 \eeq
 and
 \begin{align}\label{sec-4}
\lim_{N\to \infty}&\frac{N}{h}\intt \mbox{\rm cov}(\tilde{Z}_N(t,s), \tilde{Z}_{N,m}(t,s) )dtds\\
&=\Biggl\{  \intt \left(  \sum_{\ell=-\infty}^\infty EX_0(t)X_{\ell,m}(s)\right) \left(\sum_{j=-\infty}^\infty E X_{0,m}(t)X_j(s)\right)dtds\notag\\
  &\hspace{1 cm}   +  \left( \int \sum_{\ell=-\infty}^\infty EX_0(t)X_{\ell, m}(t)dt     \right)^2\Biggl\}\int_{-1}^1 K^2(u)du.\notag
 \end{align}
 \end{lemma}
\begin{proof}  By a simple calculation
\begin{align*}
\frac{N}{h}&E\intt \left(\tilde{Z}_N(t,s)-\tilde{Z}_{N,m}(t,s)\right)^2dtds\\
&=\frac{N}{h}\intt \mbox{var}(\tilde{Z}_N(t,s))dtds
+\frac{N}{h}\intt \mbox{var}(\tilde{Z}_{N,m}(t,s))dtds\\
&\quad -2\frac{N}{h}\intt \mbox{cov}(\tilde{Z}_N(t,s), \tilde{Z}_{N,m}(t,s))dtds,
\end{align*}
and hence \eqref{sec-1} follows from Lemma \ref{f-lem} and \eqref{sec-2}--\eqref{sec-4}.\\
We recall  $a_\ell(t,s)=EX_0(t)X_\ell(s)$ and let
\begin{align*}
\psi_{\ell,r,p}(t,s)=E[X_0(t)&X_\ell(s)X_r(t)X_p(s)]\\
&-a_\ell(t,s)a_{p-r}(t,s)-a_r(t,t)a_{p-\ell}(s,s)-a_p(t,s)a_{r-\ell}(t,s).
\end{align*}
As the first step in the proof of \eqref{sec-2} we show that
\beq\label{psi-1}
\frac{1}{h}\sum_{g,\ell=-h}^h\sum_{r=-(N-1)}^{N-1}\left|\intt \psi_{\ell, r, r+g}(t,s)dtds
\right|\to 0,\quad \mbox{as}\;\;N\to \infty.
\eeq
It is easy to see that
\begin{align}\label{psi-2}
\frac{1}{h}\sum_{g,\ell=-h}^h&\sum_{r=-(N-1)}^{N-1}\left|\intt \psi_{\ell, r, r+g}(t,s)dtds\right|\\
&=\frac{1}{h}\sum_{\ell=0}^h\sum_{g=0}^h\sum_{r=0}^{N-1}\left|\intt \psi_{\ell, r, r+g}(t,s)dtds\right|\notag \\
&\hspace{1cm}+\frac{1}{h}\sum_{\ell=0}^h\sum_{g=0}^h\sum_{r=-(N-1)}^{-1}\left|\intt \psi_{\ell, r, r+g}(t,s)dtds\right| \notag\\
&\hspace{1cm}+\cdots
+\frac{1}{h}\sum_{\ell=-h}^{-1}\sum_{g=-h}^{-1}\sum_{r=-(N-1)}^{-1}\left|\intt \psi_{\ell, r, r+g}(t,s)dtds\right|,\notag
\end{align}
where the right hand side contains eight terms corresponding to the combinations of the indices $\ell, g$ and $r$ taking either nonnegative or negative values.  Due to stationarity, we only consider the first term. In the summation of $\psi_{\ell, r, r+g}$, we consider three cases: $\ell$ is less than $r$, $\ell$ is between $r$ and $r+g$, or $\ell$ is larger than $r+g$. \\
Let $R_1=\{(\ell, g ,r): \ell<r, 0\leq \ell, g\leq h, 0\leq r \leq N-1\}, R_2=\{(\ell, g ,r): r\leq \ell \leq r+g, 0\leq \ell, g\leq h, 0\leq r \leq N-1
\},$ and $R_3=\{(\ell, g ,r):  r+g<\ell, 0\leq \ell, g\leq h, 0\leq r \leq N-1\}$. Clearly,
$$
\sum_{\ell=0}^h\sum_{g=0}^h\sum_{r=0}^{N-1}\left|\intt \psi_{\ell, r, r+g}(t,s)dtds\right|\leq U_{1,N}+U_{2,N}+U_{3, N},
$$
where
$$
U_{1,N}=\sum_{R_1}\left|\intt \psi_{\ell, r, r+g}(t,s)dtds\right|,\quad
U_{2,N}=\sum_{R_2}\left|\intt \psi_{\ell, r, r+g}(t,s)dtds\right|,
$$
and
$$
U_{3,N}=\sum_{R_3}\left|\intt \psi_{\ell, r, r+g}(t,s)dtds\right|.
$$
Using the definition of $\psi_{\ell, r, r+g}$ we write
\begin{align}\label{wro-1}
\sum_{R_1}&\left|\intt \psi_{\ell, r, r+g}(t,s)dtds\right|\\
&\leq \sum_{R_1}\left|  \intt  a_r(t,t)a_{r+g-\ell}(s,s)dtds \right|\notag\\
&\hspace{0.5 cm}+
\sum_{R_1}\left|\intt a_{r-\ell}(t,s)a_{r+g}(t,s)dtds\right|\notag\\
&\hspace{0.5 cm}+          \sum_{R_1}\left|    \intt [EX_0(t)X_\ell(s)X_r(t)X_{r+g}(s)-a_\ell(t,s)a_g(t,s)  ]dtds\right|.  \notag
\end{align}
By the inequality (A.3) in Horv\'ath and Rice (2014) and the fact that for any random variable $(E\zeta^2)^{1/2}\leq (E\zeta^4)^{1/4}$ we get that
\beq\label{wro-2}
\left|  \intt  a_r(t,t)a_{r+g-\ell}(s,s)dtds \right|\leq E\|X_0\|^2c_rc_{r+g-\ell}
\eeq
and
\beq\label{wro-3}
\left|\intt a_{r-\ell}(t,s)a_{r+g}(t,s)dtds\right|\leq E\|X_0\|^2 c_{r-\ell}c_{r+g},
\eeq
where, we recall from \eqref{as:3},
$$
c_\ell=(E\|X_0-X_{0,\ell}\|^4)^{1/4}.
$$
Combining \eqref{wro-2} with the definition of $R_1$ we conclude
\begin{align}\label{wro-4}
\sum_{R_1}\left|  \intt  a_r(t,t)a_{r+g-\ell}(s,s)dtds \right|&\leq E\|X_0\|^2\sum_{R_1} c_{r}c_{r+g-\ell}\\
&\leq E\|X_0\|^2 \sum_{\ell=0}^h\sum_{r=\ell+1}^{N-1}c_r\sum_{g=0}^hc_{r+g-\ell}\notag\\
&\leq E\|X_0\|^2 \left(\sum_{\ell=0}^\infty\sum_{r=\ell}^\infty c_r\right)\sum_{g=0}^\infty c_g.\notag
\end{align}
Similarly,
\begin{align}\label{wro-5}
\sum_{R_1}\left|\intt a_{r-\ell}(t,s)a_{r+g}(t,s)dtds\right|&\leq E\|X_0\|^2\sum_{R_1}c_{r-\ell}c_{r+g}\\
&\leq E\|X_0\|^2\sum_{\ell=0}^h\sum_{r=\ell}^{N-1}\sum_{g=0}^hc_{r-\ell}c_{r+g}\notag\\
&\leq E\|X_0\|^2 \sum_{\ell=0}^h\sum_{r=\ell}^{N-1}\sum_{p=\ell}^{\infty}c_{r-\ell}c_p\notag\\
&\leq E\|X_0\|^2 \left(\sum_{\ell=0}^\infty\sum_{p=\ell}^{\infty}c_p\right)\sum_{r=0}^{\infty}c_{r}.\notag
\end{align}
Let $1\leq \xi=\xi(N)\leq h$ be a sequence of real numbers which will be defined below. We write $R_1=R_{1,1}\cup R_{1,2}$, where $R_{1,1}=\{(\ell, g, r)\in R_1: r-\ell>\xi\}$ and $R_{1,2}=\{ (\ell, g, r)\in R_1: r-\ell\leq \xi\}$. It follows from (A.9) of Horv\'ath and Rice (2014) that there is a constant $A_1$, depending only on the distribution of $X_0$ such that for all $(\ell, g, r)\in R_{1,1}$
$$
\left|    \intt [EX_0(t)X_\ell(s)X_r(t)X_{r+g}(s)-a_\ell(t,s)a_g(t,s)  ]dtds\right|\leq A_1(c_{r-\ell}+c_{r+g-\ell}).
$$
Thus we get that
\begin{align}\label{ww-1}
\sum_{R_{1,1}} &\left|\intt  [EX_0(t)X_\ell(s)X_r(t)X_{r+g}(s)-a_\ell(t,s)a_g(t,s)  ]dtds   \right|\\
&\leq A_1\sum_{\ell=0}^h\sum_{r=\ell+\xi}^N\sum_{g=\ell-r}^h (c_{r-\ell}+c_{r+g-\ell})\notag\\
&\leq 2A_1h^2\sum_{p=\xi}^\infty c_p.\notag
\end{align}
To obtain an upper bound when the summation is over $R_{1,2}$ we write $R_{1,2}=\cup_{i=1}^3R_{1,2,i}$ where $R_{1,2,1}=\{(\ell, g ,r)\in R_{1,2}: \ell>\xi\},
R_{1,2,2}=\{(\ell, g ,r)\in R_{1,2}: g>\xi\}$ and $R_{1,2,3}=\{(\ell, g ,r)\in R_{1,2}: \ell, g\leq \xi\}$. It follows from  (A.4) in Horv\'ath and Rice (2014) that there is a constant $A_2$ depending only on $X_0$ such that
$$
\left|   \intt EX_0(t)X_\ell(s)X_r(t)X_{r+g}(s)dtds  \right|\leq A_2\min (c_\ell, c_g)\quad \mbox{for all}\quad (\ell, g, r)\in R_1.
$$
Thus we have
\begin{align}\label{ww-2}
\sum_{R_{1,2,1}}&\left|   \intt EX_0(t)X_\ell(s)X_r(t)X_{r+g}(s)dtds  \right|\\
&\leq A_2 \sum_{\ell=\xi+1}^h\sum_{r=\ell}^{\ell+\xi}\sum_{g=\ell-r}^h c_\ell \notag\\
&\leq A_2h \xi \sum_{\ell=\xi}^\infty c_\ell. \notag
\end{align}
Similarly,
\begin{align}\label{ww-3}
\sum_{R_{1,2,2}}&\left|   \intt EX_0(t)X_\ell(s)X_r(t)X_{r+g}(s)dtds  \right|
\leq A_2h \xi \sum_{\ell=\xi}^\infty c_\ell.
\end{align}
It follows from the  definitions of $R_{1,2,3}$ and $R_{1,2}$ that $R_{1,2,3}\subseteq \{0\leq \ell, g \leq \xi, 0\leq r \leq 2\xi\}$, so we have with some constant $A_3$ that
\begin{align}\label{ww-4}
\sum_{R_{1,2,3}}&\left|   \intt EX_0(t)X_\ell(s)X_r(t)X_{r+g}(s)dtds  \right|
\leq A_3  \xi^3.
\end{align}
Similar but somewhat easier arguments show
\begin{align}\label{ww-5}
\sum_{R_{1,2}}\left| \intt a_\ell(t,s)a_g(t,s)dtds\right|\leq A_4\left\{h \xi \sum_{\ell=\xi}^\infty c_\ell+
\xi^3\right\}.
\end{align}
with some constant $A_4$. If  \eqref{as:3} holds, then $\sum_{i=1}^\infty c_i<\infty$ and
\beq\label{ww-6}
\ell w\left(\sum_{i=\ell}^\infty c_i\right)\to 0,\;\;\mbox{as  }\ell\to \infty.
\eeq
Also $w^{-1}(x)$ exists for small enough $x$ (cf.\ Bingham et al (1987), pp.\ 28 and 29), and $w^{-1}(x)/x^3\to \infty$ as $x\to 0$. Using theorem 1.5.12 of Bingham et al (1987) we get that \eqref{ww-6} is equivalent to
\begin{align}\label{bs:1}
\frac{1}{w^{-1}(1/\ell)}\sum_{i=\ell}^\infty c_i\to 0,\;\;\mbox{as  }\ell\to \infty.
\end{align}
Therefore, with the choice of 
$$
\xi=\frac{1}{w(1/h)}
$$
in \eqref{ww-1}--\eqref{ww-5} we obtain that
\begin{align}\label{wro-6}
\frac{1}{h}\sum_{R_{1}} &\left|\intt  [EX_0(t)X_\ell(s)X_r(t)X_{r+g}(s)-a_\ell(t,s)a_g(t,s)  ]dtds   \right|\to 0.
\end{align}
Putting together \eqref{wro-4}, \eqref{wro-5} and \eqref{wro-6} we conclude
\beq\label{ww-10}
\frac{1}{h}\sum_{R_1}\left|\intt \psi_{\ell, r, r+g}(t,s)dtds
\right|\to 0.
\eeq
Similar arguments show that  \eqref{ww-10} remains true if the domain of summation $R_1$ is replaced with  $R_2$ or $R_3$ and hence
$$
\frac{1}{h} \sum_{\ell=0}^h\sum_{g=0}^h\sum_{r=0}^{N-1}\left| \intt \psi_{\ell, r, r+g}(t,s)dtds\right|\to 0.
$$
With minor modifications of the arguments above one can verify that the remaining seven terms in \eqref{psi-2} also tend to 0, as $N\to \infty$.\\
Now we show that \eqref{psi-1} implies \eqref{sec-2}. By a simple calculation using \eqref{as:1}  and \eqref{til-g} we get
\begin{align*}
N\mbox{cov}&(\tilde{\gamma}_\ell(t,s), \tilde{\gamma}_g(t,s))\\
&=\frac{1}{N}\Biggl\{\sum_{i=\max(1, 1-\ell)}^{\min(N,N-\ell)}\sum_{j=\max(1,1-g)}^{\min(N,N-g)}EX_i(t)X_{i+\ell}(s)X_j(t)X_{j+g}(s)\\
&\hspace{3 cm}-(N-|\ell|)(N-|g|)a_\ell(t,s)a_g(t,s)\Biggl\}\\
&=\frac{1}{N}\sum_{i=\max(1, 1-\ell)}^{\min(N,N-\ell)}\sum_{j=\max(1,1-g)}^{\min(N,N-g)}\biggl(\psi_{\ell, j-i, j-i+g}(t,s)\\
&\hspace{3 cm}+a_{j-i+g}(t,s)a_{j-i-\ell}(t,s)+a_{j-i}(t,t)a_{j-i+g-\ell}(s,s)\biggl).
\end{align*}
Notice that the summand in the last formula depends only on the difference $j-i$. Let $\varphi_N(r,\ell, g)$ denote the  cardinality of the set
$\{ (i,j): j-i=r, \max (1, 1-\ell)\leq i \leq \min (N,N-\ell), \max(1, 1-g)\leq j \leq \min(N,N-g)\}$, i.e.\ $\varphi_N(r, \ell, g)$ is the number of pairs of indices $i,j$ in the sum so that $j-i=r$. Clearly, $\varphi_N(r, \ell, g)\leq N$. Also, $\varphi_N(r, \ell, g)\geq N-2(|\ell|+|r|+|g|)$, since $\{(i, i+r): \max (|r|,1-\ell+|r|, 1-g+|r|)\leq i\leq \min (N-|r|, N-g-|r|, N-\ell-|r|)\}\subseteq \{(i,j): j-i=r, \max (1, 1-\ell)\leq i \min (N, N-\ell), \max (1, 1-g)\leq j\leq \min (N, N-g)\}$. Using the notation
\beq\label{varphidef}
\bar{\varphi}_N(r, \ell, g)=\varphi_N(r, \ell, g)/N
\eeq
 we can write
\begin{align*}
N\mbox{cov}&(\tilde{\gamma}_\ell(t,s), \tilde{\gamma}_g(t,s))\\
&=\sum_{r=-(N-1)}^{N-1}\bar{\varphi}_N(r, \ell, g)\left\{\psi_{\ell, r, r+g}(t,s)+a_{r+g}(t,s)a_{r-\ell}(t,s)+a_r(t,t)a_{r+g-\ell}(s,s)\right\}.
\end{align*}
It follows that
\begin{align*}
\frac{N}{h}\mbox{var}(\tilde{C}_N(t,s))=\frac{N}{h}\sum_{g,\ell=-h}^{h}K(g/h)K(\ell/h)\mbox{cov}&(\tilde{\gamma}_\ell(t,s), \tilde{\gamma}_g(t,s))\\
&=q_{1,N}(t,s)+q_{2,N}(t,s)+q_{3,N}(t,s),
\end{align*}
where
\begin{align*}
q_{1,N}(t,s)&=\frac{1}{h}\sum_{g,\ell=-h}^{h} \sum_{r=-(N-1)}^{N-1}K(g/h)K(\ell/h)\bar{\varphi}_N(r, \ell, g)\psi_{\ell, r, r+g}(t,s),\\
q_{2,N}(t,s)&=\frac{1}{h}\sum_{g,\ell=-h}^{h } \sum_{r=-(N-1)}^{N-1}K(g/h)K(\ell/h)\bar{\varphi}_N(r, \ell, g)a_{r+g}(t,s)a_{r-\ell}(t,s),\\
q_{3,N}(t,s)&=\frac{1}{h}\sum_{g,\ell=-h}^{h } \sum_{r=-(N-1)}^{N-1}K(g/h)K(\ell/h)\bar{\varphi}_N(r, \ell, g)a_r(t,t)a_{r+g-\ell}(s,s).
\end{align*}
We start with $q_{2,N}$. Let $\varepsilon>0$. By a change of variables we have
$$
q_{2,N}(t,s)=\frac{1}{h}\sum_{|u|,|v|\leq h+N-1}\sum_{r=b_1}^{b_2}K\left(\frac{u-r}{h}\right)K\left(\frac{v-r}{h}\right)\bar{\varphi}_N(r, r-v, u-r)a_u(t,s)a_v(t,s),
$$
where
\beq\label{b1-def}
b_1=b_1(u,v,N)=\max (u-h, v-h, -(N-1))
\eeq
and
\beq\label{b2-def}
b_2=b_2(u,v,N)=\min (u+h, v+h, N-1).
\eeq
If
$$
q_{2,N}^{(M)}(t,s)=\frac{1}{h}\sum_{|u|,|v|\leq M}\sum_{r=b_1}^{b_2}K\left(\frac{u-r}{h}\right)K\left(\frac{v-r}{h}\right)\bar{\varphi}_N(r, r-v, u-r)a_u(t,s)a_v(t,s),
$$
then we have
\begin{align*}
|q_{2,N}&(t,s)-q_{2,N}^{(M)}(t,s)|\\
&\leq \frac{1}{h}\sum_{u,v\in\Theta_{N,M}}\;\sum_{r=b_1}^{b_2}\left|K\left(\frac{u-r}{h}\right)K\left(\frac{v-r}{h}\right)\bar{\varphi}_N(r, r-v, u-r)a_u(t,s)a_v(t,s)\right|,
\end{align*}
where $\Theta_{N,M}=\{u,v: |u|, |v|\leq h+N-1, \max(|u|, |v|)\geq M\}$. By assumption \eqref{k-2}, the number of terms in $r$ such that $b_1(u,v,N)\leq r \leq b_2(u,v,N)$ and $K((u-r)/h)K((v-r)/h)\neq 0$ cannot exceed $2h$ for any
$u,v$. Since $|\bar{\varphi}_N|\leq 1$, we conclude
\beq\label{kcs0}
|q_{2,N}(t,s)-q_{2,N}^{(M)}(t,s)|\leq 2 \sup_{|x|\leq 1}K^2(x)\sum_{u,v\in\Theta_{N,M}}|a_u(t,s)a_v(t,s)|.
\eeq
The Cauchy-Schwarz inequality yields
\begin{align}\label{kcs}
\intt& |q_{2,N}(t,s)-q_{2,N}^{(M)}(t,s)|dtds\\
&\leq 2 \sup_{|x|\leq 1}K^2(x)\left[ \left( \sum_{|u|\geq M}\|a_u\|\right)^2+2\left(\sum_{|u|>M}\|a_u\|\right)\left(\sum_{v=-\infty}^\infty\|a_v\|\right)
\right]<\varepsilon/4,\notag
\end{align}
by taking $M$ sufficiently large. We recall that  $N-2(r+\ell+g)\leq \varphi(r, \ell, g)\leq N$. If $|u|, |v|\leq M, b_1(u,v,N)\leq r \leq b_2(u,v,N)$ hold,  then  $|r|\leq M+h$ and hence for such $u,v$ and $r$ we also have
$
|\varphi(r, r-v, u-r)-N|\leq 2|r+r-v+u-r|\leq 2(|r|+|u|+|v|)\leq 2(h+3M),
$
resulting in
that $|\bar{\varphi}_N(r, r-v, u-r)-1|\leq 2(3M+h)/N$.
Using  \eqref{h-1}, one can establish along the lines of the proof of \eqref{kcs}
\begin{align}\label{kcs-1}
\intt &\left|q_{2,N}^{(M)}(t,s)-\frac{1}{h}\sum_{|u|, |v|\leq M}\sum_{r=b_1}^{b_2}K\left(\frac{u-r}{h}\right)K\left(\frac{v-r}{h}\right)a_u(t,s)a_v(t,s)\right|dtds\\
&<\varepsilon/4\notag
\end{align}
for all large enough $N$. By \eqref{k-3} and \eqref{h-1}, for any $\eta>0$ we have
$$
\sup_{|u|,|v|\leq M}\sup_r \left|K\left(\frac{u-r}{h}\right)K\left(\frac{v-r}{h}\right)-K^2\left(\frac{r}{h}\right)
\right|<\eta,
$$
when $N$ is sufficiently large. Since we can take $\eta>0$ as small as we wish, it holds for all large enough $N$ that
\begin{align}\label{kcs-3}
\frac{1}{h}\sum_{|u|,|v|\leq M} &\sum_{r=b_1}^{b_2}\intt \left|\left[ K\left(\frac{u-r}{h}\right)K\left(\frac{v-r}{h}\right)-K^2\left(\frac{r}{h}\right)\right]
a_u(t,s)a_v(t,s)\right|dtds\\
&<\varepsilon/4.\notag
\end{align}
Clearly, according to the definition of a Riemann integral
$$
\sup_{|u|, |v|\leq M}\left|\frac{1}{h}\sum_{r=b_1(u,v,N)}^{b_2(u,v,N)}K^2\left(\frac{r}{h}\right)-\int_{-1}^1 K^2(z)dz\right|\to 0,\quad\mbox{as    } N\to \infty.
$$
Using the definition of $C(t,s)$ one can easily see via the Lebesgue dominated convergence theorem that
\begin{align*}
\intt \sum_{|u|,|v|\leq M}a_u(t,s)a_v(t,s)dtds=\intt \left(\sum_{|u|\leq M}a_u(t,s)\right)^2dtds\to \|C\|^2,
\end{align*}
as $M\to \infty$. Thus we get that for all $N$ and $M$ sufficiently large
\begin{align}\label{rie-1}
\Biggl|\frac{1}{h} \intt & \sum_{|u|,|v|\leq M} \sum_{r=b_1}^{b_2}K^2\left(\frac{r}{h}\right)a_u(t,s)a_v(t,s)dtds -\|C\|^2\int_{-1}^1 K^2(z)dz\Biggl|
\leq \varepsilon/4.
\end{align}
Combining \eqref{kcs0}--\eqref{rie-1} we conclude
\beq\label{li-2}
\lim_{N\to\infty}\intt q_{2,N}(t,s)dtds=\|C\|^2\int_{-1}^1 K^2(z)dz.
\eeq
Observing that
$$
\intt \sum_{|u|, |v|\leq M}a_u(t,t)a_v(s,s)dtds=\left(\int \sum_{|u|\leq M} a_u(t,t)dt \right)^2\;\;\to\;\;
\left(\int C(t,t)dt\right)^2,
$$
as $M\to \infty$,
minor modifications of the proof of \eqref{li-2} yield
\beq\label{li-3}
\lim_{N\to\infty}\intt q_{3,N}(t,s)dtds=\left(\int C(t,t)dt\right)^2\int_{-1}^1 K^2(z)dz.
\eeq
Finally, by \eqref{psi-1}
\beq\label{li-1}
\left|\intt q_{1,N}(t,s)dtds\right|\leq \frac{1}{h}\sup_{|x|\leq c}K^2(x) \sum_{g,\ell=-h}^h\sum_{r=-(N-1)}^{N-1}\left|\intt \psi_{\ell, r, r+g}(t,s)dtds\right|
\to 0,
\eeq
as $N\to \infty$. The result in  \eqref{sec-2}  now follows from \eqref{li-2}--\eqref{li-1}.\\
Clearly, \eqref{sec-3} is a special case of \eqref{sec-2}.\\
Let
\begin{align*}
\psi_{\ell, r, p}^{(m)}(t,s)=EX_0(t)X_\ell(s)&X_{r,m}(t)X_{p,m}(s)-a_\ell(t,s)a_{p-r,m}(t,s)\\
&-a^{(2)}_{r,m}(t,t)a_{p-\ell, m}^{(2)}(s,s)-a^{(2)}_{p,m}(t,s)a^{(1)}_{r-\ell, m}(t,s),
\end{align*}
where
$$a_{\ell, m}(t,s)=EX_{0,m}X_{\ell,m}(s), a_{\ell,m}^{(1)}(t,s)=EX_{0,m}(t)X_\ell(s)\;\; \mbox{and}\;\;a_{\ell,m}^{(2)}(t,s)=EX_{0}(t)X_{\ell, m}(s).
 $$
 Under the conditions of the Theorem \ref{main-1} we have that
\begin{align}\label{mdep-1}
\frac{1}{h}\sum_{g,\ell=-h}^h\sum_{r=-(N-1)}^{N-1}\left|\intt \psi_{\ell, r, r+g}^{(m)}(t,s)dtds\right|\;\;\to\;\; 0,\quad{as}\;\;\;N\to \infty,
\end{align}
 along the lines of \eqref{psi-1}.   It follows from the definitions of $\tilde{Z}_N$ and
$\tilde{Z}_{N,m}$ that
\begin{align*}
\frac{N}{h}E\tilde{Z}_N(t,s)\tilde{Z}_{N,m}(t,s)=\frac{N}{h}\sum_{\ell=-h}^h\sum_{k=-h}^hK\left(\frac{\ell}{h}\right)K\left(\frac{k}{h}\right)
\mbox{cov}(\tilde{\gamma}_\ell(t,s), \tilde{\gamma}_{g,m}(t,s)).
\end{align*}
Also,
\begin{align*}
N\mbox{cov}&(\tilde{\gamma}_\ell(t,s), \tilde{\gamma}_{g,m}(t,s))\\
&=\frac{1}{N}\Biggl\{\sum_{i=\max(1, 1-\ell)}^{\min(N,N-\ell)}\sum_{j=\max(1,1-g)}^{\min(N,N-g)}\psi^{(m)}_{\ell, j-i, j-i+g}(t,s)\\
&\hspace{3 cm}+a_{j-i+g,m}^{(2)}(t,s)a_{j-i-\ell,m}^{(1)}(t,s)+a_{j-i,m}^{(2)}(t,t)a_{j-i+g-\ell}^{(2)}(s,s)\Biggl\}.
\end{align*}
Following the proof of \eqref{li-1} one can show that \eqref{mdep-1} implies
\begin{align*}
\frac{1}{h}\left|\intt  \sum_{\ell=-h}^h\sum_{k=-h}^h \sum_{i=\max(1, 1-\ell)}^{\min(N,N-\ell)}\sum_{j=\max(1,1-g)}^{\min(N,N-g)}
K\left(\frac{\ell}{h}\right)K\left(\frac{k}{h}\right) \psi^{(m)}_{\ell, j-i, j-i+g}(t,s)dtds  \right|\;\to\;0,
\end{align*}
as $N\to \infty$. Along the lines of \eqref{li-2} and \eqref{li-3} we get that
\begin{align*}
\frac{1}{h}&\sum_{\ell=-h}^h\sum_{k=-h}^h \sum_{i=\max(1, 1-\ell)}^{\min(N,N-\ell)}\sum_{j=\max(1,1-g)}^{\min(N,N-g)}
K\left(\frac{\ell}{h}\right)K\left(\frac{k}{h}\right)\intt a_{j-i+g,m}^{(2)}(t,s)a_{j-i-\ell,m}^{(1)}(t,s)dtds\\
&\to\;\; \intt \left(  \sum_{\ell=-\infty}^\infty EX_0(t)X_{\ell,m}(s)\right) \left(\sum_{j=-\infty}^\infty E X_{0,m}(t)X_j(s)\right)dtds\int_{-1}^1 K^2(u)du
\end{align*}
and
\begin{align*}
\frac{1}{h}&\sum_{\ell=-h}^h\sum_{k=-h}^h \sum_{i=\max(1, 1-\ell)}^{\min(N,N-\ell)}\sum_{j=\max(1,1-g)}^{\min(N,N-g)}
K\left(\frac{\ell}{h}\right)K\left(\frac{k}{h}\right)\intt a_{j-i,m}^{(2)}(t,t)a_{j-i+g-\ell}^{(2)}(s,s)dtds\\
 &\to\;\;   \left( \int \sum_{\ell=-\infty}^\infty EX_0(t)X_{\ell, m}(t)dt     \right)^2\int_{-1}^1 K^2(u)du,
 \end{align*}
 completing  the proof of \eqref{sec-4}.
\end{proof}

\subsection{Approximations with finite dimensional processes}\label{fine-dim} Based on the result in Section \ref{lem-mdep},    we now assume that
\beq\label{dim-1}
X_i(t), -\infty<i <\infty \;\mbox{is an  } m\mbox{--dependent stationary sequence},
\eeq
\beq\label{dim-2}
EX_i(t)=0\;\;\mbox{and}\;\;E\|X_0\|^4<\infty.
\eeq
First we replace the bivariate cumulant function $\psi_{\ell, r, p}$  of Section \ref{lem-mdep} with the four variate version
\begin{align*}
\chi_{\ell, r, p}(t,s, \tp, \spr)=E[X_0(t)X_\ell(s)X_r(\tp)X_p(\spr)]&-a_\ell(t,s)a_{p-r}(\tp, \spr)-a_r(t,\tp)a_{p-\ell}(s,\spr)\\
&-a_p(t,\spr)a_{r-\ell}(s,\tp).
\end{align*}
\begin{lemma}\label{appr-le-1} If \eqref{dim-1} and  \eqref{dim-2} are satisfied, then we have
$$
\left\|\sum_{\ell, r, p=-\infty}^\infty \chi_{\ell, r, p}(t,s,\tp, \spr)\right\|<\infty.
$$
\end{lemma}
\begin{proof} Using stationarity arguments, we need to prove only that
$$
\left\|\sum_{\ell, r, p=0}^\infty \chi_{\ell, r, p}(t,s,\tp, \spr)\right\|<\infty.
$$
Let $D=\{(\ell, r, p): \ell, r, p\geq 0\}$ and $D_1=\{(\ell, r, p): 0\leq \ell\leq  r\leq p\}$. If $(\ell, r, p)\in D_1$ and $p-r>m$, then $\chi_{\ell, r, p}=0$ since each term in the definition of $\chi_{\ell, r, p}$ equals $0$ in this case due to the $m$--dependence. Similarly, if $\ell>m$ or $r-\ell>m$,  then $\chi_{\ell, r, p}$ equals $0$ for all $(\ell, r, p)\in D_1$.  
Therefore $\{(\ell, r,p)\in D_1: \chi_{\ell, r, p}\neq 0\}\subseteq \{(\ell, r, p):0\leq  \ell\leq m, 0\leq r\leq 2m, p-r\leq m\}$ and the last set has no more than $6(m+1)^3$ elements. Hence
\beq\label{mstep-1}
\left\|\sum_{D_1} \chi_{\ell, r, p}(t,s,\tp, \spr)\right\|<\infty
\eeq
since only finitely many terms are different from zero in the sum. The other subsets of $D$ can be handled similarly so the details are omitted.
\end{proof}
\medskip

Let
$$
L_N(t,s,\tp, \spr)=\frac{N}{h}\sum_{\ell, g=-h}^{h}K\biggl(\frac{\ell}{h}\biggl)K\biggl(\frac{g}{h}\biggl)\mbox{var}(\tilde{\gamma}_\ell(t,s), \tilde{\gamma}_g(\tp,\spr)).
$$
\begin{lemma}\label{appr-le-2} If \eqref{dim-1} and  \eqref{dim-2} are satisfied, then we have that
$$
\|L_N-L\|\to 0.
$$
where $L$ is defined in \eqref{L-def}.
\end{lemma}
\begin{proof}
Following the proof of Lemma \ref{sec-lem} we write
\begin{align*}
L_N(t,s,\tp, \spr)=q_{1,N}(t,s,\tp, \spr)+q_{2,N}(t,s,\tp, \spr)+q_{3,N}(t,s,\tp, \spr),
\end{align*}
where
\begin{align*}
q_{1,N}(t,s,\tp, \spr)&=\frac{1}{h}\sum_{g,\ell=-h}^{h}\sum_{r=-(N-1)}^{N-1} K\left(\frac{\ell}{h}\right)K\biggl(\frac{g}{h}\biggl)\bar{\varphi}_N(r, \ell, g) \chi_{\ell, r, r+g}(t,s,\tp, \spr),\\
q_{2,N}(t,s,\tp, \spr)&=\frac{1}{h}\sum_{g,\ell=-h}^{h}\sum_{r=-(N-1)}^{N-1} K\left(\frac{\ell}{h}\right)K\biggl(\frac{g}{h}\biggl)\bar{\varphi}_N(r, \ell, g) a_{r+g}(t, \spr)a_{r-\ell}(\tp, s),\\
q_{3,N}(t,s,\tp, \spr)&=\frac{1}{h}\sum_{g,\ell=-h}^{h}\sum_{r=-(N-1)}^{N-1} K\left(\frac{\ell}{h}\right)K\biggl(\frac{g}{h}\biggl)\bar{\varphi}_N(r, \ell, g) a_{r}(t, \tp)a_{r+g-\ell}( s, \spr),
\end{align*}
and $\bar{\varphi}_N(r, \ell, g)$ is defined an \eqref{varphidef}. If  we write
$$
L(t,s,\tp, \spr)=L^{(1)}(t,s,\tp, \spr)+L^{(2)}(t,s,\tp, \spr),
$$
where
$$
L^{(1)}(t,s,\tp, \spr)=C(t,s)C(\tp, \spr) \int_{-c}^c K^2(z)dz
$$
and
$$
L^{(2)}(t,s,\tp, \spr)=C(t,\tp)C(s,\spr) \int_{-c}^c K^2(z)dz,
$$
then by the triangle inequality we get
$$
\|L_N-L\|\leq \|q_{1,N}\|+\|q_{2,N}-L^{(1)}\|+\|q_{3,N}-L^{(2)}\|.
$$
Clearly,
$$
\|q_{1,N}\|\leq \frac{1}{h}\sup_{|x|\leq c}K^2(x)\left\| \sum_{\ell, r, q=-\infty}^\infty \chi_{\ell, r, q}\right\|\;\;\to\;\;0,\;\;\mbox{as}\;\;N\to \infty,
$$
on account of Lemma \ref{appr-le-1}. By assumption \eqref{h-1} for all large enough $N$ we have that $h+N-1 \geq m$. Since $a_j=0$ for all $|j|>m$,  by a change of variables we have for all large enough $N$
\begin{align*}
q_{2,N}(t,s,\tp, \spr)&=\frac{1}{h}\sum_{|u|, |v|<h+N}\sum_{r=b_1}^{b_2}K\left(\frac{u-r}{h}\right)K\left(\frac{v-r}{h}\right)\bar{\varphi}_N(r,r-v,u-r)a_u(t,\spr)a_v(\tp, s)\\
&=\frac{1}{h}\sum_{|u|, |v|\leq m}\sum_{r=b_1}^{b_2}K\left(\frac{u-r}{h}\right)K\left(\frac{v-r}{h}\right)\bar{\varphi}_N(r,r-v,u-r)a_u(t,\spr)a_v(\tp, s),
\end{align*}
where $b_1=b_1(u,v,N)$ and $b_2=b_2(u,v,N)$ are defined in \eqref{b1-def} and \eqref{b2-def}, respectively. Define
\begin{align*}
q_{2,N,1}(t,s,\tp, \spr)&=\frac{1}{h}\sum_{|u|, |v|\leq m}\sum_{r=b_1}^{b_2}K\left(\frac{u-r}{h}\right)K\left(\frac{v-r}{h}\right)a_u(t,\spr)a_v(\tp, s),\\
q_{2,N,2}(t,s,\tp, \spr)&=\frac{1}{h}\sum_{|u|, |v|\leq m}\sum_{r=b_1}^{b_2}K^2\left(\frac{r}{h}\right)a_u(t,\spr)a_v(\tp, s),\\
q_{2,N,3}(t,s,\tp, \spr)&=\frac{1}{h}\sum_{|u|, |v|\leq m}\sum_{r=-h}^{h}K^2\left(\frac{r}{h}\right)a_u(t,\spr)a_v(\tp, s).
\end{align*}
Since $|\bar{\varphi}_N(r,r-v,u-r)-1|\leq 2(3m+h)/N$ for all $1\leq r \leq N$ and $|u|, |v|\leq m$ we conclude by Fubini's theorem that as $N\to \infty$,
\begin{align*}
\|q_{2,N}-q_{2,N,1}\|
\leq \frac{2Q(3m+h)}{N}\sup_{|x|\leq 1}K^2(x)
\;\;\;\to\;\;0,
\end{align*}
where
$$
Q=\left(\intt \left(\sum_{|u|\leq m}|a_u(t,s')|\right)^2dtds'\intt\left(\sum_{|v|\leq m}|a_v(t',s)|\right)^2dt'ds\right)^{1/2}.
$$
Using  \eqref{k-3} one can find a constant $\eta$ such that
\beq\label{qq-k1}
\sup_{|u|\leq m}\left|K\left(\frac{u-r}{h}\right)-K\left(\frac{r}{h}\right)\right|\leq \eta\frac{m}{h}
\eeq
and therefore
\beq\label{qq-k2}
\sup_{|u|, |v|\leq m}\left|K\left(\frac{u-r}{h}\right)K\left(\frac{v-r}{h}\right)-K^2\left(\frac{r}{h}\right)\right|\leq 2\eta\frac{m}{h}\sup_{|x|\leq 1}|K(x)|.
\eeq
By \eqref{qq-k1} and \eqref{qq-k2} we obtain that
$$
\|q_{2,N,1}-q_{2,N,2}\|\leq 2\eta Q\frac{m}{h}\sup_{|x|\leq 1}|K(x)|\;\;\to\;\;0,\;\;\mbox{ as }\;\;N\to \infty.
$$
It follows from the definitions of $b_1=b_1(u,v,N)$ and  $b_2=b_2(u,v,N)$  in \eqref{b1-def} and \eqref{b2-def} that
$$
\|q_{2,N,2}-q_{2,N,3}\|\leq \frac{2Qm}{h}\sup_{|x|\leq 1}K^2(x)\;\;\to\;\;0,\;\;\mbox{ as }\;\;N\to \infty.
$$
Finally,
$$
\|q_{2,N,3}-L^{(1)}\|\leq Q\left|\frac{1}{h}\sum_{r=-h}^{h}K^2\left(\frac{r}{h}\right)-\int_{-1}^1 K^2(z)dz
\right|\;\;\to\;\;0,
$$
as $N\to\infty$, since $K$ is Riemann integrable. This also concludes the proof of
$$
\|q_{2,N}-L^{(1)}\|\;\;\to\;\;0,\;\;\mbox{ as }\;\;N\to \infty.
$$
Similar arguments yield
$$
\|q_{3,N}-L^{(2)}\|\;\;\to\;\;0,\;\;\mbox{ as }\;\;N\to \infty,
$$
completing the proof of Lemma \ref{appr-le-2}.
\end{proof}
\medskip
Let $\phi_i(t), 1\leq i <\infty$ be an orthonormal basis of $L^2([0,1],\Rn)$. By the Karhunen--Lo\'eve expansion we can write
$$
X_i(t)=\sum_{\ell=1}^\infty \la X_i, \phi_\ell\ra \phi_\ell(t).
$$
Define
$$
X_i^{(d)}(t)=\sum_{\ell=1}^d \la X_i, \phi_\ell\ra \phi_\ell(t)
$$
and correspondingly $\bar{Z}^{(d)}_N=\bar{C}^{(d)}_N(t,s)-E\bar{C}^{(d)}_N(t,s)$, where
\begin{align*}
\bar{C}^{(d)}_N(t,s)=\sum_{i=-\infty}^\infty K\left(\frac{i}{h}\right)\bar{\gamma}_i^{(d)}(t,s)
\end{align*}
with
\begin{displaymath}
\bar{\gamma}_i^{(d)}(t,s)=\tilde{\gamma}_{i,N}(t,s)=
\left\{
\begin{array}{ll}
\displaystyle \frac{1}{N}\sum_{j=1}^{N-i}X_j^{(d)}(t)X_{j+i}^{(d)}(s),\quad &i\geq 0\\
\vspace{.3cm}
\displaystyle \frac{1}{N}\sum_{j=1-i}^N X_j^{(d)}(t)X_{j+i}^{(d)}(s), \quad &i<0.
\end{array}
\right.
\end{displaymath}
It follows from the Karhunen--Lo\'eve theorem  that
\beq\label{d-01}
E\| X_0-X_0^{(d)}\|^2\to 0,\;\;\mbox{as}\;\;d\to \infty.
\eeq
Let
$$
C^{(d)}(t,s)=\sum_{\ell=-m}^m EX^{(d)}_0(t)EX^{(d)}_\ell(s).
$$
\begin{lemma}\label{appr-le-3} If \eqref{dim-1} and  \eqref{dim-2} are satisfied, then we have that
\beq\label{appr-4}
\lim_{d\to \infty}\limsup_{N\to \infty}\frac{N}{h}E\|\tilde{Z}_N-\bar{Z}_N^{(d)}\|^2=0.
\eeq
Also, for each $d \geq 1$,
 \beq\label{appr-31}
\lim_{N\to \infty}\frac{N}{h}\intt \mbox{\rm var}(\bar{Z}_{N}^{(d)}(t,s))dtds  =\left(\|C^{(d)}\|^2 +\left(\int C^{(d)}(t,t)dt\right)^2 \right)\int_{-c}^c K^2(u)du  ,
 \eeq
 and
 \begin{align}\label{appr-32}
\lim_{N\to \infty}&\frac{N}{h}\intt \mbox{\rm cov}(\tilde{Z}_N(t,s), \tilde{Z}_{N}^{(d)}(t,s) )dtds\\
&=\Biggl\{  \intt \left(  \sum_{\ell=-m}^m EX_0(t)X_{\ell}^{(d)}(s)\right) \left(\sum_{j=-m}^m E X_{0}^{(d)}(t)X_\ell(s)\right)dtds\notag\\
  &\hspace{1 cm}   +  \left( \int \sum_{\ell=-m}^m EX_0(t)X_{\ell}^{(d)}(t)dt     \right)^2\Biggl\}\int_{-c}^c K^2(u)du.\notag
 \end{align}
\end{lemma}
\begin{proof}
Using the Cauchy--Schwarz inequality and stationarity we have
\begin{align}\label{cau-1}
&\intt \Biggl|E[(X_0(t)-X_0^{(d)}(t))X_i(s)]EX_0(t)X_j(s)\Biggl|dtds\\
&\leq \left\{\intt [E(X_0(t)-X_0^{(d)}(t))X_i(s)]^2dtds\right\}^{1/2}\left\{\intt [ EX_0(t)EX_j(s)]^2  dtds\right\}^{1/2}  \notag\\
&\leq \left\{\intt E(X_0(t)-X_0^{(d)}(t))^2EX_i^2(s)dtds\right\}^{1/2}
\left\{\intt  EX^2_0(t)EX_j^2(s)  dtds\right\}^{1/2}   \notag\\
&=(E\|X_0-X_0^{(d)}\|^2)^{1/2}(E\|X_0\|^2)^{3/2}\notag
\end{align}
and similarly
\begin{align}\label{cau-2}
\intt &\left| E[X_0^{(d)}(t)(X_i(s)-X_i^{(d)}(s))]EX_0^{(d)}(t)X_j^{(d)}(s)\right|dtds\\
&\leq (E\|X_0-X_0^{(d)}\|^2)^{1/2}(E\|X_0^{(d)}\|^2)^{3/2}\notag\\
&\leq (E\|X_0-X_0^{(d)}\|^2)^{1/2}(E\|X_0\|^2)^{3/2}. \notag
\end{align}
Hence by elementary calculations we conclude from these inequalities
\begin{align*}
\intt &\left|EX_0(t)X_i(s)EX_0(t)X_j(s)-EX_0^{(d)}(t)X_i^{(d)}(s)EX_0^{(d)}(t)X_j^{(d)}(s)
\right|dtds\\
 &\leq A (E\|X_0-X_0^{(d)}\|^2)^{1/2}(E\|X_0\|^2)^{3/2}.
\end{align*}
with some constant $A$.
Thus we get as $d\to \infty$ that
\begin{align}\label{d-11}
\left\| \sum_{\ell=-m}^m EX_0^{(d)}(t) EX_\ell^{(d)}(s) \right\|\;\;\to\;\;\left\| \sum_{\ell=-m}^m EX_0(t) EX_\ell(s) \right\|,
\end{align}
\begin{align}\label{d-12}
\int \sum_{\ell=-m}^m EX_0^{(d)}(t) EX_\ell^{(d)}(t)dt\;\;\to\;\;\int \sum_{\ell=-m}^m EX_0(t) EX_\ell(t)dt
\end{align}
\begin{align}\label{d-13}
\intt &\left(  \sum_{\ell=-m}^m EX_0(t)X_{\ell}^{(d)}(s)\right) \left(\sum_{j=-m}^m E X_{0}^{(d)}(t)X_\ell(s)\right)dtds\\
&\quad \quad\;\;\to\;\;\left\| \sum_{\ell=-m}^m EX_0(t) EX_\ell(s) \right\|,\notag
\end{align}
and
\begin{align}\label{d-13a}
\int \sum_{\ell=-m}^m EX_0(t)X_{\ell}^{(d)}(t)dt   \;\;\to\;\;\int \sum_{\ell=-m}^m EX_0(t) EX_\ell(t)dt.
\end{align}
On account of \eqref{d-01}, the result in \eqref{appr-4} follows from \eqref{sec-2}, \eqref{appr-31}, \eqref{appr-32} and \eqref{d-11}--\eqref{d-13a}.\\
Lemma \ref{appr-le-2} implies \eqref{appr-31}.
The proof of  \eqref{appr-32} goes along the lines of \eqref{sec-4} but it is much simpler since \eqref{psi-1} always satisfied for $m$--dependent random functions. Hence the details are omitted.
\end{proof}

\subsection{Normal approximation   in case of  finite dimensional $m$--dependent processes}\label{clt-fine-dim} Based on the result in Sections \ref{lem-mdep} and \ref{fine-dim}    we can and will assume in this section that
\beq\label{findim-1}
X_i(t), -\infty<i <\infty \;\mbox{is an  } m\mbox{--dependent} \;d\mbox{--dimensional stationary sequence},
\eeq
i.e.
\beq\label{findim-2}
X_i(t)=\sum_{\ell=1}^d\la X_i, \phi_\ell \ra\phi_\ell(t),
\eeq
where $\phi_\ell(t), \ell\geq 1$ is a basis of $L^2$. Let
\begin{align*}
C_N^*(t,s)=\sum_{i=-\infty}^\infty K\left(\frac{i}{h}\right){\gamma}^*_i(t,s)
\end{align*}
with for all $-N<i<N$
\begin{align*}
{\gamma}_i^*(t,s)=
\displaystyle \frac{1}{N}\sum_{j=1}^{N}X_j(t)X_{j+i}(s).
\end{align*}
First we show that the difference between $Z_N^*(t,s)=C_N^*(t,s)-EC_N^*(t,s)$ and $\tilde{Z}_N(t,s)$ is small.
\medskip

\begin{lemma}\label{fini-0} If \eqref{findim-1},\eqref{k-1}--\eqref{h-1} are satisfied,  then we have

$$
\frac{N}{h}\|Z_N^*-\tilde{Z}_N\|^2=o_P(1),\quad\mbox{as}\quad N\to\infty.
$$
\end{lemma}
\begin{proof} Let
\begin{displaymath}
s_{\ell, N}=\left\{
\begin{array}{ll}
\{j:\;N-\ell<j\leq N\},\quad\mbox{if   } \ell\geq 0
\vspace{.3 cm}\\
\{j:\;1\leq j\leq 1-\ell\},\quad\mbox{if   } \ell< 0.
\end{array}
\right.
\end{displaymath}
Then according to the definitions of $\tilde{Z}_N$ and $Z_N^*$ we have
\begin{align}\label{pi-1}
E\|\tilde{Z}_N-Z_N^*\|^2
=\frac{1}{N^2}&\intt \sum_{\ell=-\infty}^\infty\sum_{p=-\infty}^\infty K\left(\frac{\ell}{h}  \right)K\left(\frac{p}{h}  \right)\\
&\times
\sum_{j\in s_{\ell, N}}  \sum_{i\in s_{p,N}}\biggl( EX_j(t)X_{j+\ell}(s)X_i(t)X_{i+p}(s)
-a_\ell(t,s)a_p(t,s)\biggl)dtds. \notag
\end{align}
Using the $m$--dependence of the $X_i$'s, one can verify along the lines of the arguments used in Lemma \ref{sec-lem} that the right side of \eqref{pi-1} is $O(h^2)$.
Thus the result follows from \eqref{h-1} via Markov's inequality.
\end{proof}

Using  \eqref{findim-2} we have
\begin{align*}
{Z}^*_N(t,s)&=\sum_{\ell=-h}^h K\left(\frac{\ell}{h}\right)({\gamma}^*_\ell(t,s)-E{\gamma}^*_\ell(t,s))\\
&=
\sum_{r=1}^d\sum_{p=1}^d\left\{\frac{1}{N}\sum_{\ell=-h}^h\sum_{j=1}^N(\xi_{r,j}\xi_{p, j+\ell}-E\xi_{r,j}\xi_{p, j+\ell})K(\ell/h)\right\}\phi_r(t)\phi_p(s),
\end{align*}
where
$$
\xi_{r,j}=\la X_j, \phi_r\ra.
$$
In order to  show that $(N/h)^{1/2}{Z}_N^*(t,s)$ can be approximated with a Gaussian process, we begin by establishing that the $d^2$--dimensional vector
\begin{align}\label{multi-1}
\left\{\frac{1}{(Nh)^{1/2}}\sum_{\ell=-h}^h\sum_{j=1}^N(\xi_{r,j}\xi_{p, j+\ell}-E\xi_{r,j}\xi_{p, j+\ell})K(\ell/h), 1\leq r,p\leq d\right\}\stackrel{{\mathcal D}}{\to}
{\calN}_{d^2},
\end{align}
where ${\calN}_{d^2}$ is a $d^2$--dimensional normal random vector. By the Cram\'er--Wold device it is sufficient to show that
\begin{align}\label{clt-1}
\sum_{r=1}^d\sum_{p=1}^d\beta_{r,p}\frac{1}{(Nh)^{1/2}}\sum_{\ell=-h}^h\sum_{j=1}^N(\xi_{r,j}\xi_{p, j+\ell}-E\xi_{r,j}\xi_{p, j+\ell})K(\ell/h)\stackrel{{\mathcal D}}{\to}
{\calN},
\end{align}
for any constants $\beta_{r,p}, 1\leq r,p\leq d$, where $\calN$ denotes a normal random variable. By the definition of $\xi_{i,j}$,
$$
\xi_{r,j}\xi_{p, j+\ell}=\intt X_{j}(t)X_{j+\ell}(s)\phi_r(t)\phi_p(s)dtds,
$$
and therefore
\begin{align*}
\sum_{r=1}^d\sum_{p=1}^d\beta_{r,p}&\frac{1}{(Nh)^{1/2}}\sum_{\ell=-h}^h\sum_{j=1}^N(\xi_{r,j}\xi_{p, j+\ell}-E\xi_{r,j}\xi_{p, j+\ell})K(\ell/h)\\
&=\frac{1}{(Nh)^{1/2}}\sum_{\ell=-h}^hK\left(\frac{\ell}{h}\right)\sum_{j=1}^N\intt \biggl\{\bigl(X_j(t)X_{j+\ell}(s)\\
&\hspace{3cm}-EX_j(t)X_{j+\ell}(s)\bigl)\sum_{r=1}^d\sum_{p=1}^d\beta_{r,p}\phi_r(t)\phi_p(s)
\biggl\}dtds.
\end{align*}
Therefore \eqref{clt-1} follows  if we prove that for any $f\in L^2([0,1]^2,\Rn)$
\begin{align}\label{clt-2}
\frac{1}{(Nh)^{1/2}}\sum_{\ell=-h}^hK\left(\frac{\ell}{h}\right) \sum_{j=1}^N\alpha_{j,\ell}\stackrel{{\mathcal D}}{\to}
{\calN},
\end{align}
where 
$$
\alpha_{j,\ell}=\intt (X_{j}(t)X_{j+\ell}(s)-EX_{j}(t)X_{j+\ell}(s))f(t,s)dtds.
$$
The proof of \eqref{clt-2} is based on a blocking argument. We write
$$
\Delta_N=\sum_{\ell=-h}^h\sum_{j=1}^N \alpha_{j,\ell}K\left(\frac{\ell}{h}\right)=\sum_{i=1}^QR_i+\sum_{i=1}^QD_i+T',
$$
where
$$
R_i=\sum_{j\in B_i}\sum_{\ell=-h}^h\alpha_{j,\ell}K\left(\frac{\ell}{h}\right),\quad D_i=\sum_{j\in b_i}\sum_{\ell=-h}^h\alpha_{j,\ell}K\left(\frac{\ell}{h}\right)\quad\mbox{and} \quad T'=\sum_{j\in T}\sum_{\ell=-h}^h\alpha_{j,\ell}K\left(\frac{\ell}{h}\right)
$$
with
\begin{align*}
B_i&=\left\{j: 1+(i-1)M+2(i-1)h\leq j<1+iM+2(i-1)h   \right\},\;\;Q=\lfloor N/(2h+M)\rfloor\\
b_i&=\left\{j: 1+iM+2(i-1)h\leq j <1+i(M+2h)\right\}, \;\; T=\{j: Q(2h+M)\leq j\leq N\}
\end{align*}
and $M>h$ is a numerical sequence. It follows from assumption \eqref{findim-1} that $R_1, R_2, \ldots , R_Q$ are independent and identically distributed random variables with zero mean. Similarly, $D_1, D_2, \ldots , D_Q$ are independent and identically distributed random variables with zero mean.\\

The following elementary lemma will be useful to get sharp upper bounds for the moments of the blocks.

\begin{lemma}\label{momcom} Suppose $\{Y_i\}_{i=1}^\infty$ is an $m$--dependent sequence of random variables. Then for all $\ell\geq 0$, the sequence of random vectors $\{(Y_i, Y_{i+\ell})\}_{i=1}^\infty$ can be organized into at most $3(m+1)^2$ collections each containing independent random vectors.
\end{lemma}
\begin{proof} The sequence $\{(Y_i, Y_{i+\ell})\}_{i=1}^\infty$ is $m+\ell$--dependent, it can be organized into $m+\ell +1$ subsets, each containing independent random variables using standard arguments (see, for example, Lemma 2.4 of Berkes et al (2012)).  Hence the result is proven for $\ell\leq 2m+1$. From now on we assume that $\ell>2m+1$. Let $j^*=\max\{j:\;j<(\ell-m)/(m+1)\}$,
$k^*=\min\{k:\;k>(m+\ell)/((m+1)j^*)\}$, and $v^*=((k^*+1)j^*+1)(m+1).$ Define the set
$$
G_{i,k}(p)=\biggl\{(Y_{pv^*+(kj^*+j)(m+1)+i}, Y_{pv^*+(kj^*+j)(m+1)+i+\ell}), 0\leq j\leq j^*\biggl\}
$$
for $1\leq i \leq m+1$,  $0\leq k\leq k^*$ and $0\leq p <\infty$. Consider two arbitrary elements of $G_{i,k}(p)$,
$\X_r=(Y_{pv^*+(kj^*+r)(m+1)+i}, Y_{pv^*+(kj^*+r)(m+1)+i+\ell})$ and $\X_t=(Y_{pv^*+(kj^*+t)(m+1)+i},\\
 Y_{pv^*+(kj^*+t)(m+1)+i+\ell})$, where, without loss of generality, $0\leq r<t\leq j^*$. Clearly, $Y_{pv^*+(kj^*+r)(m+1)+i}$ is independent of $Y_{pv^*+(kj^*+t)(m+1)+i}$ and $Y_{pv^*+(kj^*+t)(m+1)+i+\ell}$, since $r<t$. Also, $Y_{pv^*+(kj^*+r)(m+1)+i+\ell})$ is independent of $Y_{pv^*+(kj^*+t)(m+1)+i+\ell})$ due to $r<t$. Using the definition of $j^*$, we have
\begin{align*}
\left|pv^*+(kj^*+r)(m+1)+i+\ell-(pv^*+(kj^*+t)(m+1)+i)\right|
&=\left|\ell-(r-t)(m+1)\right|\\
&>\ell-j^*(m+1)>m.
\end{align*}
 Hence $Y_{pv^*+(kj^*+r)(m+1)+i+\ell}$ and $Y_{pv^*+(kj^*+t)(m+1)+i}$ are independent, establishing the independence of $\X_r$ and $\X_t$. It follows along these lines that the vectors in $G_{i,k}(p)$ are mutually independent.
 Due to the definition of $j^*$, $G_{i,k}(p)$ is comprised of $j^*+1$ independent random variables. Further,  according  to the definition of $v^*$, $G_{i,k}(p)$ and $G_{i,k}(p')$ are independent for all integers $p\neq p'$. By the definitions of $k^*$  and $j^*$ we have that
$$
k^*\leq \frac{m+\ell}{(m+1)j^*}\leq \frac{m+\ell}{\ell-2m-1}\leq 3m+2,\quad
\mbox{where we used}\quad
j^*\geq \frac{\ell-m}{m+1}-1.
$$
It follows that the union of the at most $(3m+3)(m+1)$ sets
$\cup_{p=0}^\infty G_{i,k}(p), 1\leq i \leq m+1, 0\leq k \leq k^*$ contains  $\{(Y_i, Y_{i+\ell})\}_{i=1}^\infty$.

\end{proof}

\begin{lemma}\label{fini-1} If \eqref{findim-1},\eqref{k-1}--\eqref{k-3} are satisfied, $h=h(N)\to \infty$ and $M/h\to \infty$, then we have
\beq\label{ff-1}
E\left((hM)^{-1/2}R_1\right)^2\to\int\!\!\!\cdots\!\!\!\int L(t,s, t',s') f(t,s)f(t',s')dtdsdt'ds'
\eeq
where $L$ is defined in \eqref{L-def},
\beq\label{ff-2}
E\left|R_1\right|^{2+\delta/2}=O(h^{2+\delta/2}M^{1+\delta/4})
\eeq
and
\beq\label{ff-3}
ED_1^2=O(h^2).
\eeq
\end{lemma}
\begin{proof} The assertions in \eqref{ff-1} and \eqref{ff-3} can be established along the lines of the proof Lemma \ref{appr-le-2}. Due to the $m$--dependence assumed in \eqref{findim-1}, the proofs are much simpler in the present case. \\
By Petrov (1995, p.\ 58)
$$
E\left|R_1\right|^{2+\delta/2}\leq (2h+1)^{1+\delta/2}\sum_{\ell=-h}^hK^{2+\delta/2}\left(\frac{\ell}{h}\right)E\left( \sum_{i=1}^{M+1}\alpha_{i,\ell} \right)^{2+\delta/2}.
$$
Using Lemma \ref{momcom} we can write   $\sum_{i=1}^{M +1}\alpha_{i,\ell}$  as the sum of no more than $3(m+1)^2$ sums, each sum is based on i.i.d. random variables.  Hence via the triangle and Rosenthal's inequalities we get
$$
\left(E\left( \sum_{i=1}^{M+1} \alpha_{i,\ell} \right)^{2+\delta/2}\right)^{1/(2+\delta/2)}\leq c_0 {M}^{1/2}
$$
with some constant $c_0$, completing the proof of \eqref{ff-2}.
\end{proof}
\begin{lemma}\label{fini-2} If \eqref{findim-1} and \eqref{k-1}--\eqref{m-1}  are satisfied, then we have
$
(Nh)^{-1/2}\Delta_N
$
converges in distribution to a normal random variable with zero mean and variance\\
 $\int\!\!\cdots\!\!\int L(t,s,\tp, \spr)f(t,s)f(\tp,\spr)dtdsd\tp d\spr.$
\end{lemma}
\begin{proof}   Under assumption \eqref{m-1} one can find a sequence $M$ such that $M/h\to \infty$ and $h(M/N)^{\delta/(4+\delta)}\to 0$ and therefore using \eqref{ff-3} of Lemma \ref{fini-1} and the independence of the $D_i$'s we obtain that
$$
E\left(\frac{1}{\sqrt{Nh}}\sum_{i=1}^QD_i\right)^2=O\left(\frac{1}{Nh}Qh^2   \right)=O\left( \frac{h}{M}\right)=o(1)
$$
and therefore
$$
\frac{1}{\sqrt{Nh}}\sum_{i=1}^QD_i\stackrel{P}{\to}0.
$$
Similar argument yields
$$
\frac{1}{\sqrt{Nh}}T'\stackrel{P}{\to}0.
$$
Using now \eqref{ff-1} and \eqref{ff-2} we conclude
\begin{align*}
\frac{\left(\sum_{i=1}^QE|R_i|^{2+\delta/2}\right)^{1/(2+\delta/2)}}{\left(\sum_{i=1}^QER_i^2\right)^{1/2}}&=O(1)\frac{(Qh^{2+\delta/2}M^{1+\delta/4})^{1/(2+\delta/2)}}
{(hQM)^{1/2}}\\
&=O(1)\frac{N^{1/(2+\delta/2)}hM^{\delta/(8+2\delta)}}{(Nh)^{1/2}}\to 0.
\end{align*}
 Now Lyapunov's theorem (cf.\ Petrov (1995, p.\ 126) and \eqref{ff-1} imply Lemma \ref{fini-2}.
\end{proof}
\begin{lemma}\label{fini-3} If \eqref{findim-1},\eqref{k-1}--\eqref{m-1} and  are satisfied, then we can define Gaussian processes $\Gamma_N(t,s)$ with
$E\Gamma_N(t,s)=0$, $E\Gamma_N(t,s)\Gamma_N(\tp,\spr)=L(t,s,\tp, \spr)$ such that
$$
\|(N/h)^{1/2}{Z}_N^*-\Gamma_N\|=o_P(1),\quad \mbox{as}\;\;N\to \infty.
$$
\end{lemma}
\begin{proof} As we argued at the beginning of this section, Lemma \ref{fini-2} yields that \eqref{multi-1} holds with ${\calN}_{d^2}=\{ \calN_{d^2}(r,p), 1\leq r, p\leq d\}$, $E\calN_{d^2}(r,p)=0$ and
$$
E\calN_{d^2}(r,p)\calN_{d^2}(r',p')=\int\!\!\!\cdots\!\!\!\int L(t,s,\tp, \spr)\phi_r(t)\phi_p(s) \phi_{r'}(\tp)\phi_{p'}(\spr)dtdsd\tp d\spr.
$$
By the Skorokhod--Dudley--Wichura   representation (cf.\ Shorack and Wellner (1986), p.\ 47) we can define ${\calN}^{(N)}_{d^2}=\{ \calN^{(N)}_{d^2}(r,p), 1\leq r,p\leq d$, a copy of ${\mathcal N}_{d^2}$ such that
\begin{align}\label{multi-3}
\max_{1\leq r,p\leq d}\left|{(Nh)^{-1/2}}\sum_{\ell=-h}^h\sum_{j=1}^N(\xi_{r,j}\xi_{p, j+\ell}-E\xi_{r,j}\xi_{p, j+\ell})K(\ell/h)- \calN^{(N)}_{d^2}(r,p)\right|
=o_P(1).
\end{align}
Clearly,
$$
\Gamma_N(t,s)=\sum_{r=1}^d\sum_{p=1}^d\calN^{(N)}_{d^2}(r,p)\phi_r(t)\phi_p(s)
$$
is a Gaussian process with mean zero and $E\Gamma_N(t,s)\Gamma_N(\tp,\spr)=L(t,s,\tp, \spr)$. Using now \eqref{multi-3}, Lemma \ref{fini-3} follows.
\end{proof}

Next we show if \eqref{m-2} is satisfied then the conclusion of Lemma \ref{fini-3} holds assuming only \eqref{h-1} instead of the much stronger restriction \eqref{m-1} on $h$.

\begin{lemma}\label{fini-4} If \eqref{findim-1},\eqref{k-1}--\eqref{h-1} and  \eqref{m-2} are satisfied, then we can define Gaussian processes $\Gamma_N(t,s)$ with
$E\Gamma_N(t,s)=0$, $E\Gamma_N(t,s)\Gamma_N(\tp,\spr)=L(t,s,\tp, \spr)$ such that
$$
\|(N/h)^{1/2}{Z}_N^*-\Gamma_N\|=o_P(1),\quad \mbox{as}\;\;N\to \infty.
$$
\end{lemma}

\begin{proof} Following the proof of Lemma \eqref{appr-le-2} we can verify that
\beq\label{mo-1}
ER_1^4=O(h^2M^2).
\eeq
In the proof of Lemma \ref{fini-2} when  Lyapunov's condition is verified we now use \eqref{mo-1} instead of \eqref{ff-2}. Hence Lemma \ref{fini-2} holds assuming only \eqref{h-1} when \eqref{m-2} holds. Now Lemma \ref{fini-4} follows from the central limit theorem of Lemma \ref{fini-2} using the Skorokhod--Dudley--Wichura  representation theorem (cf.\ Shorack and Wellner (1986), p.\ 47) as in Lemma \ref{fini-3}.
\end{proof}

\subsection{Proofs of Theorems \ref{main-1}--\ref{bias-1}}\label{proof-1}
\noindent
{\it Proof of Theorem \ref{main-1}.} Using the results in Sections \ref{lem-mdep} and \ref{fine-dim} with Lemma \ref{fini-3}, for every $\varepsilon>0$ there are integers $m_0$ and $d_0$ such that
$$
\limsup_{N\to \infty}P\left\{ \|(N/h)^{1/2}Z_N-\Gamma_{N,d,m}\|>\varepsilon  \right\}<\varepsilon \quad \mbox{for all}\quad m>m_0, d>d_0,
$$
where
$$
\{\Gamma_{N,d,m}(t,s), 0\leq t,s\leq 1\}\stackrel{{\mathcal D}}{=}\left\{\sum_{r=1}^d\sum_{p=1}^d {\mathcal N}^{(m)}_{d^2}(r,p)\phi_r(t)\phi_p(s),
0\leq t,s\leq 1
\right\},
$$

$$
\Gamma_{d,m}(t,s)=\sum_{r=1}^d\sum_{p=1}^d {\mathcal N}^{(m)}_{d^2}(r,p)\phi_t(t)\phi_p(s),
$$

$\{\phi_i\}_{i=1}^\infty$ is an orthonormal basis of $L^2$, and $\{{\mathcal N}^{(m)}_{d^2}(r,p), 1\leq r,p\leq d\}$ is $d^2$--dimensional normal with zero mean
and
$$
E{\mathcal N}^{(m)}_{d^2}(r,p){\mathcal N}^{(m)}_{d^2}(r',p')=\int\!\!\!\cdots\!\!\!\int L^{(m)}(t,s,\tp, \spr)\phi_r(t)\phi_p(s) \phi_{r'}(\tp)\phi_{p'}(\spr)dtdsd\tp d\spr
$$
with
$$
L^{(m)}(t,s,\tp, \spr)=\left[C_m(t,s)C_m( \tp, \spr) +C_m(t,\tp)C_m(s,\spr) \right] \int_{-1}^1 K^2(z)dz,
$$
$$
C_m(t,s)=\sum_{\ell=-m}^m\mbox{cov}(X_{0,m}(t), X_{\ell,m}(s)),
$$
and the variables $X_{i,m}$ are defined in \eqref{as:3}. Using \eqref{fi-1} we conclude $\|L^{(m)}-L\|\to 0$, as $m\to \infty$ and therefore
$$
\{{\mathcal N}^{(m)}_{d^2}(r,p), 1\leq r,p\leq d\}\;\;\stackrel{{\mathcal D}}{\to}\;\;\{{\mathcal N}(r,p), 1\leq r,p\leq d\},\quad\mbox{as   }m\to\infty,
$$
where $\{{\mathcal N}(r,p), 1\leq r,p<\infty\}$ is Gaussian with zero mean and
$$
E{\mathcal N}(r,p){\mathcal N}(r',p')=\int\!\!\!\cdots\!\!\!\int L(t,s,\tp, \spr)\phi_r(t)\phi_p(s) \phi_{r'}(\tp)\phi_{p'}(\spr)dtdsd\tp d\spr.
$$
Hence we can define Gaussian processes $\Gamma_d^{(m)}(t,s)$ such that $\|\Gamma_{d,m}-\Gamma_d^{(m)}\}=o_P(1)$,
$\{\Gamma_d^{(m)}(t,s), 0\leq t,s \leq 1\} \stackrel{{\mathcal D}}{=}\{\Gamma_d(t,s), 0\leq t,s \leq 1\}$ and
$$
\Gamma_{d}(t,s)=\sum_{r=1}^d\sum_{p=1}^d {\mathcal N}(r,p)\phi_t(t)\phi_p(s).
$$
Observing that $\|\Gamma_d-\Gamma\|=o_P(1)$ as $d\to\infty$, where $\Gamma(t,s)=\sum_{r=1}^\infty\sum_{p=1}^\infty {\mathcal N}(r,p)\phi_t(t)\phi_p(s)$,
the proof of Theorem \ref{main-1} is complete.
\qed

\medskip
\noindent
{\it Proof of Theorem \ref{main-2}.} We can repeat the proof of Theorem \ref{main-1}, we only need to replace Lemma \ref{fini-3} with
Lemma \ref{fini-4}.
\qed

\medskip
\noindent
{\it Proof of Theorem \ref{bias-1}.}  It is easy to see that
\begin{align*}
E\hat{C}_N(t,s)=\sum_{\ell=-\infty}^\infty K\left(\frac{\ell}{h}\right)\gamma_\ell(t,s)-\frac{1}{N}\sum_{\ell=-\infty}^\infty K\left(\frac{\ell}{h}\right)\ell\gamma_\ell(t,s)
\end{align*}
and
$$
\left\|\sum_{\ell=-\infty}^\infty  K\left(\frac{\ell}{h}\right)\ell\gamma_\ell  \right\|=O(1).
$$
Let $\varepsilon>0$. Next we write
\begin{align*}
\sum_{\ell=-\infty}^\infty K\left(\frac{\ell}{h}\right)\gamma_\ell(t,s)=\sum_{\ell=-\infty}^\infty\gamma_\ell(t,s)+f_{N,1,\varepsilon}(t,s)+f_{N,2,\varepsilon}(t,s)-f_{N,3,\varepsilon}(t,s),
\end{align*}
where
$$
f_{N,1,\varepsilon}(t,s)=\sum_{|\ell|\leq \varepsilon h}\left(K\left(\frac{\ell}{h}\right)-1\right)\gamma_\ell(t,s), \quad \quad f_{N,2,\varepsilon}(t,s)=\sum_{|\ell|>\varepsilon h} K\left(\frac{\ell}{h}\right)\gamma_\ell(t,s)
$$
and
$$
f_{N,3,\varepsilon}(t,s)=\sum_{|\ell|>\varepsilon h} \gamma_\ell(t,s).
$$
Using assumption \eqref{par-2}  we conclude by the triangle inequality
$$
\|f_{N,2,\varepsilon}(t,s)\|\leq \sup_u|K(u)|\sum_{|\ell|>\varepsilon h} \|\gamma_\ell\|\leq (h\varepsilon)^{-{\frak q}'}\sup_u|K(u)|\sum_{\ell=-\infty} ^\infty |\ell|^{{\frak q}'} \|\gamma_\ell\|
$$
and similarly
$$
\|f_{N,3,\varepsilon}(t,s)\|\leq(h\varepsilon)^{-{\frak q}'}\sum_{\ell=-\infty} ^\infty |\ell|^{{\frak q}'} \|\gamma_\ell\|.
$$
By \eqref{par-1}  we obtain that
$$
\lim_{\varepsilon\to 0}\limsup_{h\to \infty}\max_{-\varepsilon h\leq \ell\leq  \varepsilon h }\left|(K(\ell/h)-1)/(|\ell|/h)^{\frak q}-{\frak K}\right|=0
$$
and therefore
$$
\lim_{\varepsilon\to 0}\limsup_{h\to \infty} \left\| h^{\frak q}f_{N,1,\varepsilon}-{\frak F}\right\|=0.
$$
Since ${\frak q}'>{\frak q}$ and $h^{\frak q}/N\to 0$, the proof of Theorem \ref{bias-1} is complete.
\qed

\subsection{Proof of Theorem \ref{th:eig-1} }\label{sec-eig-pr}
Following the arguments used in the proof of Proposition 1 of Kokoszka and Reimherr (2013) one can show that
\beq\label{g-1}
\max_{1\leq \ell\leq p}\|(N/h)^{1/2}(\hat{s}_\ell\hat{v}_\ell-v_\ell)-\hat{\mathcal G}_{\ell, N}\|=o_P(1),
\eeq
where
\begin{align*}
\hat{\mathcal G}_{\ell, N}=\sum_{k\neq \ell}\frac{v_k(t)}{\lambda_\ell-\lambda_k}\intt Z_N^*(u,s)v_\ell(u)v_k(s)duds
\end{align*}
and
$$
Z_N^*(u,s)=(N/h)^{1/2}(\hat{C}_N(u,s)-C(u,s)).
$$
By Theorems \ref{main-1} and \ref{bias-1} we have
\beq\label{g-2}
\max_{1\leq \ell\leq p}\|Z_N^*-(\Gamma_N+{\frak a}{\frak F})\|=o_P(1).
\eeq
Let ${\frak M}_\ell $ be the mapping from $L^2[0,1]^2\to L^2[0,1]$ defined by
$$
{\frak M}_\ell (f)(t)=\sum_{k\neq \ell }\frac{v_k(t)}{\lambda_\ell-\lambda_k}\intt f(u,s)v_\ell(u)v_k(s)duds.
$$
The linear operators ${\frak M}_\ell, 1\leq \ell\leq p$ are bounded since
\begin{align*}
\|{\frak M}_\ell (f)\|^2=\sum_{k\neq \ell}(\lambda_\ell-\lambda_k)^{-2}\left(\intt f(u,s)v_\ell(s)v_k(u)dsdu\right)^2\leq \frac{\|f\|^2}{\alpha_\ell},
\end{align*}
where
\begin{displaymath}
\alpha_\ell=\left\{
\begin{array}{ll}
\lambda_1-\lambda_2, &\quad\mbox{if  } \ell=1
\vspace{.3 cm}\\
\min\{\lambda_{\ell-1}-\lambda_\ell, \lambda_\ell-\lambda_{\ell+1}\},&\quad\mbox{if  } 2\leq \ell\leq p.
\end{array}
\right.
\end{displaymath}
The operators ${\frak M}_\ell, 1\leq \ell\leq p$ are linear and bounded and therefore they are continuous (cf.\ Debnath and  Mikusinski (1999, p.\ 27)). Hence \eqref{g-1} and  \eqref{g-2} imply that
$$
\max_{1\leq \ell\leq p}\|(N/h)^{1/2}(\hat{s}_\ell\hat{v}_\ell-v_\ell)-{\mathcal G}_{\ell, N}\|=o_P(1)
$$
where
$$
{\mathcal G}_{\ell,N}(t)=
\sum_{k\neq \ell}\frac{v_k(t)}{\lambda_\ell-\lambda_k}\intt (\Gamma_N(u,s)+{\frak a}{\frak F}(u,s))v_\ell(u)v_k(s)duds.
$$
Similar but somewhat simpler arguments give
$$
\max_{1\leq \ell\leq p}|(N/h)^{1/2}\hat{\lambda}_\ell-\lambda_\ell-{\frak g}_{\ell, N}|=o_P(1),
$$
where
$$
{\frak g}_{\ell, N}=\intt (\Gamma_N(t,s)+{\frak a}{\frak F}(t,s))v_\ell(t)v_\ell(s)dtds.
$$
Since $v_\ell(t)v_k(s)$ is a basis in $L^2([0,1]^2,\Rn)$, by the Karhunen--Lo\'eve expansion we have for each $N$
\beq\label{root}
\{\Gamma_N(t,s), 0\leq t,s\leq 1\}\stackrel{{\mathcal D}}{=}\left\{\sum_{1\leq i, j<\infty}\sigma_{i, j}^{1/2}{\mathcal N}_{i, j}v_i(t)v_j(s), 0\leq t,s\leq 1
\right\},
\eeq
where
\begin{displaymath}\label{rep-last}
\sigma_{i,j}=\left\{
\begin{array}{ll}
\displaystyle \lambda_i\lambda_j\int_{-c}^cK^2(z)dz,\quad &\mbox{if   }i\neq j,
\vspace{.3 cm}\\
\displaystyle 2\lambda_i^2\int_{-c}^cK^2(z)dz,\quad &\mbox{if   }i= j.
\end{array}
\right.
\end{displaymath}
The representation of the limit in Theorem \ref{th:eig-1} follows from \eqref{root} and the definitions of ${\frak g}_{\ell, N}$ and ${\mathcal G}_{\ell,N}(t)$.

\end{document}